%% file: tube.tex
\documentclass[twoside]{amsart}
\usepackage{amsmath,amsfonts,amsthm,mathrsfs,amscd}
\usepackage[dpi=1270,PostScript=dvips,heads=LaTeX]{diagrams}
\diagramstyle[labelstyle=\scriptstyle,small,thin]

\usepackage[unicode]{hyperref}
\usepackage{enumitem}
\usepackage[numeric,initials]{amsrefs}

\input{abbreviations.tex}

\newarrow{Equal}   =====
\newarrow{IntoBold}   {boldhook}---> 
\newtheorem{theorem}{Theorem}
\newtheorem{lemma}[theorem]{Lemma}
\newtheorem*{lemma*}{Lemma}
\newtheorem*{lefschetz-theorem}{Lefschetz Hyperplane Theorem}
\newtheorem*{theorem*}{Theorem}

\newtheorem{proposition}[theorem]{Proposition}
\theoremstyle{definition}

\theoremstyle{remark}
\newtheorem*{remark}{Remark}
\newtheorem*{example}{Example}

\newtheorem*{question*}{Question}

\theoremstyle{plain}

\newcommand{\prim}{_{\mathit{prim}}}
\newcommand{\van}{_{\mathit{van}}}

\newcommand{\famX}{\mathfrak{X}}
\newcommand{\famXsm}{\mathfrak{X}^{\mathit{sm}}}
\newcommand{\Psm}{P^{\mathit{sm}}}
\newcommand{\pism}{\pi^{\mathit{sm}}}
\newcommand{\pisml}{\pi_{\ast}^{\mathit{sm}}}
\newcommand{\unitint}{\lbrack 0, 1 \rbrack}

\DeclareMathOperator{\Alg}{Alg}

\renewcommand{\pair}[2]{B(#1, #2)}
\newcommand{\Pair}[2]{B \bigl( #1, #2 \bigr)}
\DeclareMathOperator{\Aut}{Aut}
\DeclareMathOperator{\Sp}{Sp}
\DeclareMathOperator{\Spsh}{Sp^{\sharp}}
\DeclareMathOperator{\Sptsh}{Sp_2^{\sharp}}

\renewcommand{\Ext}{\operatorname{Ext}}
\renewcommand{\Hom}{\operatorname{Hom}}

\newcommand{\il}{i_{\ast}}
\newcommand{\fl}{f_{\ast}}
\newcommand{\iu}{i^{\ast}}
\newcommand{\qu}{q^{\ast}}

\newcommand{\Phiu}{\Phi^{\ast}}

\newcommand{\itu}{i_t^{\ast}}

\renewcommand{\argbl}{-}
\newcommand{\locsys}{\mathcal{M}}
\newcommand{\dualX}{X^{\vee}}

\newcommand{\gZZ}{g^{\ZZ}}
\newcommand{\Nab}{N^{\mathit{ab}}}

\begin{document}

%========================================================
\title{Primitive cohomology and the tube mapping}
\author[C.~Schnell]{Christian Schnell}
\address{Department of Mathematics, Statistics \& Computer Science \\
University of Illinois at Chicago \\
851 South Morgan Street \\
Chicago, IL 60607}
\email{cschnell@math.uic.edu}

\subjclass[2000]{14F45; 20J06; 14D05}
\keywords{Primitive cohomology, Lefschetz pencil, Monodromy, Group cohomology}

\begin{abstract}
Let $X$ be a smooth complex projective variety of dimension $d$. We show that its
primitive cohomology in degree $d$ is generated by certain ``tube classes,''
constructed from the monodromy in the family of all hyperplane sections of $X$. The proof
makes use of a result about the group cohomology of certain representations that
may be of independent interest.
\end{abstract}
\maketitle
%========================================================

\section{Introduction}
\label{sec:introduction}

Let $X$ be a complex projective manifold.  When $X$ is embedded into
projective space, there is a close relationship between the cohomology of $X$
and that of any smooth hyperplane section $S = X \cap H$; this is the content of the
Lefschetz Hyperplane Theorem. In fact, the only piece of the cohomology of $X$ that
cannot be inferred from that of $S$ is the \define{primitive cohomology} in degree $d = \dim
X$, 
\[
	H^d(X, \QQ) \prim 
		= \ker \bigl( H^d(X, \QQ) \to H^d(S, \QQ) \bigr),
\]
which consists of those $d$-th cohomology classes on $X$ that restrict to zero on any
smooth hyperplane section.  Similarly, all the cohomology of $S$ is 
determined by that of $X$, except for the \define{vanishing cohomology}
\[
	H^{d-1}(S, \QQ) \van = 
		\ker \bigl( H^{d-1}(S, \QQ) \to H^{d+1}(X, \QQ) \bigr).
\]
It takes its name from the fact that it corresponds, under Poincar\'e duality, to the
space of homology classes in $H_{d-1}(S, \QQ)$ that vanish when mapped to $H_{d-1}(X, \QQ)$.

The content of this paper is that the primitive cohomology of $X$ can be obtained, at
least topologically, from the family of all smooth hyperplane sections, in a way that
we now describe. Hyperplane sections of $X$ (for the given embedding) are naturally
parametrized by a projective space $P$, and smooth hyperplane sections by a
Zariski-open subset $\Psm$. As is well-known, any two smooth hyperplane sections are
isomorphic as smooth manifolds, making it possible to transport homology classes
among nearby ones. Since $\Psm$ is typically not simply connected, this leads to a
\define{monodromy action} of its fundamental group $G$ on the homology $H_{\ast}(S_0,
\QQ)$ of any given smooth hyperplane section $S_0$.

The flat transport of homology classes can also be used to produce elements of
$H_d(X, \QQ)$. Namely, suppose a homology class $\alpha \in H_{d-1}(S_0, \QQ)$ is
invariant under the action of some element $g \in G$. When $\alpha$ is transported
along a closed path representing $g$, it moves through a one-dimensional family of
hyperplane sections, and in the process, traces out a $d$-chain on $X$. This
$d$-chain is actually a $d$-cycle, because $g \cdot \alpha = \alpha$. Taking the
ambiguities in this construction into account, we get a well-defined element of
$H_d(X, \QQ) / H_d(S_0, \QQ)$; we shall call it the \define{tube class}
determined by $g$ and $\alpha$. 

Under Poincar\'e duality, the quotient $H_d(X, \QQ) / H_d(S_0, \QQ)$ is
isomorphic to the primitive cohomology of $X$; the above construction thus gives us
the \define{tube mapping}
\[
	\bigoplus_{g \in G} 
		\menge{\alpha \in H^{d-1}(S_0, \QQ) \van }{g \cdot \alpha = \alpha}
		\to H^d(X, \QQ) \prim.
\]
The main result of the paper is that the tube mapping is \emph{surjective}, provided that
the necessary condition $H^{d-1}(S_0, \QQ) \van \neq 0$ is satisfied. This gives a
positive answer to a question by H.~Clemens. When the
dimension of $X$ is odd, the condition holds for essentially any embedding of $X$
into projective space; when $d$ is even, it holds as long as the degree of the
embedding is sufficiently high. Evidently, this situation can always be achieved by composing
with a suitable Veronese map.

In the remainder of the introduction, we review the Lefschetz theorems, which
describe the relationship between the cohomology of $X$ and that of a smooth hyperplane
section.  We then give a more careful definition of the tube mapping, and state the
main result in Theorem~\ref{thm:tube-main}.

\subsection*{Review of the Lefschetz theorems}

A comprehensive discussion of the relationship between the cohomology of $X$ and that
of a smooth hyperplane section $S = X \cap H$ can be found, for instance, in the book
by C.~Voisin \cite{Voisin}*{Section~13}. We only give a very brief outline of
the main points. Let us write $i \colon S \to X$ for the inclusion map; we also let
$d = \dim X$ be the complex dimension of $X$. The complement $X \setminus S$ is 
a Stein manifold, and Morse theory shows that it has the homotopy type of a
$d$-dimensional CW-complex. One consequence is the following.

\begin{lefschetz-theorem}
The restriction map $\iu \colon H^k(X, \ZZ) \to H^k(S, \ZZ)$ is an isomorphism
for $k < \dim S = d - 1$, and injective for $k = d - 1$. Moreover, the quotient group
$H^{d-1}(S, \ZZ) / H^{d-1}(X, \ZZ)$ is torsion-free. 
\end{lefschetz-theorem}

Since Poincar\'e duality on $X$ (resp.\@ $S$) can be used to describe the cohomology
groups in dimensions greater than $d$ (resp.\@ $d-1$), there are only two pieces of
the cohomology rings of $X$ and $S$ that are not covered by Lefschetz' theorem. One
is the \define{primitive cohomology}
\begin{align*}
	H^d(X, \ZZ) \prim 
		&= \ker \bigl( \iu \colon H^d(X, \ZZ) \to H^d(S, \ZZ) \bigr) \\
		&= \ker \bigl( L \colon H^d(X, \ZZ) \to H^{d+2}(X, \ZZ) \bigr),
\end{align*}
where $L$ is the Lefschetz operator, given by cup product with the fundamental class
of $S$ in $H^2(X, \ZZ)$.
The other is the quotient $H^{d-1}(S, \ZZ) / H^{d-1}(X, \ZZ)$. It is related to the 
\define{vanishing cohomology} of the hypersurface
\[
	H^{d-1}(S, \ZZ) \van = 
		\ker \bigl( \il \colon H^{d-1}(S, \ZZ) \to H^{d+1}(X, \ZZ) \bigr).
\]
The vanishing cohomology is Poincar\'e dual to the kernel of $\il \colon H_{d-1}(S, \ZZ) \to
H_{d-1}(X, \ZZ)$, and it is known that the latter is generated by the vanishing
cycles of any Lefschetz pencil on $X$, thus explaining the name.

At least over $\QQ$, one has direct sum decompositions
\begin{equation} \label{eq:decompS}
	H^{d-1}(S, \QQ) = \iu H^{d-1}(X, \QQ) \oplus H^{d-1}(S, \QQ) \van
\end{equation}
and
\begin{equation} \label{eq:decompX}
	H^d(X, \QQ) = \il H^{d-2}(S, \QQ) \oplus H^d(X, \QQ) \prim,
\end{equation}
orthogonal with respect to the intersection pairings on $S$ and $X$,
respectively. This is part of the content of the so-called Hard Lefschetz Theorem
\cite{Voisin}*{Proposition~14.27 on p.~328}. With integer coefficients, the map
\[
	H^{d-1}(S, \ZZ) \van \to H^{d-1}(S, \ZZ) / H^{d-1}(X, \ZZ)
\]
is unfortunately neither injective nor surjective in general.

% In a later section, we will need to use coefficients in $\ZZ$ as well, and so we note the
% following well-known fact.
%  
% \begin{lemma}
% The intersection pairing $H^{d-1}(S, \ZZ) \tensor H^{d-1}(S, \ZZ) \to \ZZ$ is
% nondegenerate and unimodular. In other words, the induced map
% \[
% 	H^{d-1}(S, \ZZ) \to \Hom_{\ZZ} \bigl( H^{d-1}(S, \ZZ), \ZZ \bigr)
% \]
% is surjective, and its kernel is torsion.
% \end{lemma}
% \begin{proof}
% By Poincar\'e duality, taking the cap product with the fundamental class of $S$
% defines an isomorphism
% \[
% 	\PD \colon H^{d-1}(S, \ZZ) \to H_{d-1}(S, \ZZ).
% \]
% On the other hand, the Universal Coefficients Theorem gives an exact sequence
% \begin{diagram}
% 	0 &\rTo& \Ext_{\ZZ}^1 \bigl( H_{d-2}(S, \ZZ), \ZZ \bigr) &\rTo& H^{d-1}(S, \ZZ) &\rTo&
% 		\Hom_{\ZZ} \bigl( H_{d-1}(S, \ZZ), \ZZ \bigr) &\rTo& 0,
% \end{diagram}
% where the last map is given by the Kronecker product $H^{d-1}(S, \ZZ) \tensor
% H_{d-1}(S, \ZZ) \to \ZZ$. The intersection pairing $H^{d-1}(S, \ZZ) \tensor
% H^{d-1}(S, \ZZ) \to \ZZ$ is then given by the formula
% \[
% 	(\alpha, \beta) = \bigl\langle \alpha \cup \beta, \lbrack S \rbrack \bigr\rangle = 
% 		\bigl\langle \alpha, \beta \cap \lbrack S \rbrack \bigr\rangle = 
% 		\bigl\langle \alpha, \PD(\beta) \bigr\rangle.
% \]
% Both assertions now follow from the exact sequence above.
% \end{proof}

\subsection*{The tube mapping} 
\label{subsec:tube-mapping-def}

By definition, the primitive cohomology of $X$ cannot be obtained from a single
smooth hyperplane section. But as we have seen, there is a way to produce primitive
cohomology classes on $X$, using the family of all smooth hyperplane sections. We now
give a more precise description of this process.

Let $P$ be the projective space that parametrizes all hyperplane sections of $X$ inside
the ambient projective space. The incidence variety
\[
	\famX = \menge{(H, x) \in P \times X}{x \in X \cap H} \subseteq P \times X,
\]
together with the projection $\pi \colon \famX \to P$, is called the
\emph{universal hyperplane section}, since its fiber over a point $H \in P$ is precisely
$X \cap H$. Note that $\famX$ is itself a smooth and very ample hypersurface in $P
\times X$. Those hyperplanes $H$ for which $X \cap H$ is
smooth form a Zariski-open subset $\Psm \subseteq P$, and by restricting $\pi$, we
obtain the family $\pism \colon \famXsm \to \Psm$ of all smooth hyperplane sections
of $X$.

Now fix a base point $H_0 \in \Psm$, and let $S_0 = X \cap H_0$ be the corresponding
hypersurface in $X$. The fundamental group $G = \pi_1 \bigl( \Psm, H_0 \bigr)$ of
$\Psm$ then acts by monodromy on the homology $H_{d-1}(S_0, \ZZ)$ of the fiber. (In
all other degrees, the action is trivial because of the Hyperplane Theorem.)

Whenever a $(d-1)$-cycle $\alpha \in H_{d-1}(S_0, \ZZ)$ is invariant under the action of an
element $g \in G$, we can use it to produce a homology class in $H_d(X, \ZZ)$, as follows.
The element $g$ can be represented by an immersion $\SSn{1} \to \Psm$. Transporting
$\alpha$ flatly along this closed path, and taking the trace in $X$, we get a $d$-chain
$\Gamma$ whose boundary $\partial \Gamma = g \cdot \alpha - \alpha$ is zero in
homology. By virtue of the exact sequence
\begin{diagram}
	\dotsb &\rTo& H_d(S_0, \ZZ) &\rTo^{\il}& H_d(X, \ZZ) &\rTo& H_d(X, S_0, \ZZ)
		&\rTo^{\partial}& H_{d-1}(S_0, \ZZ) &\rTo& \dotsb,
\end{diagram}
we can lift $\Gamma$ to a well-defined element $\tau_g(\alpha) \in H_d(X, \ZZ) / \il
H_d(S_0, \ZZ)$. This element does not depend on which representatives are chosen for
$\alpha$ and $g$, and we get a map
\[
	\bigoplus_{g \in G} \menge{\alpha \in H_{d-1}(S_0, \ZZ)}{g \cdot \alpha = \alpha}
		\to H_d(X, \ZZ) / \il H_d(S_0, \ZZ).
\]

We can use Poincar\'e duality on $S_0$ and on $X$ to obtain a map in cohomology; but to
get primitive cohomology classes on $X$, we need to use the decomposition in
\eqref{eq:decompX}, which only works with rational coefficients. It implies a
canonical isomorphism
\[
	H^d(X, \QQ) \prim \simeq H^d(X, \QQ) / \il H^{d-2}(S_0, \QQ) 
		\simeq H_d(X, \QQ) / \il H_d(S_0, \QQ).
\]
Restricting to the vanishing cohomology of $S_0$, we now obtain
the \define{tube mapping} in its final form as
\begin{equation} \label{eq:tube-def}
	\bigoplus_{g \in G} \menge{\alpha \in H^{d-1}(S_0, \QQ) \van}{g \cdot \alpha = \alpha} 
		\to H^d(X, \QQ) \prim.
\end{equation}
The main result of this paper is that this mapping is surjective, provided the
left-hand side is nontrivial to begin with.

\begin{theorem} \label{thm:tube-main}
Let $X$ be a smooth complex projective variety of dimension $d$, with a given
embedding into projective space. As above, let $\Psm$ be the set of
hyperplanes $H$ such that $X \cap H$ is smooth.  Let $S_0 = X \cap H_0$ be the hypersurface
corresponding to some base point $H_0 \in \Psm$, and write $G = \pi_1 \bigl( \Psm, H_0
\bigr)$ for the fundamental group of $\Psm$. 
If $H^{d-1}(S_0, \QQ) \van \neq 0$, then the tube mapping in \eqref{eq:tube-def} is
surjective.
\end{theorem}

\section{An application of the result to Calabi-Yau threefolds}

A concrete interpretation of Theorem~\ref{thm:tube-main} is as follows. Consider the
\'etale space $T \van$ of the local system on $\Psm$, whose fiber over a point corresponding
to the hyperplane section $S = X \cap H$ is the group $H^{d-1}(S, \ZZ) \van$. Points
of $T \van$ can naturally be viewed as pairs $(S, \alpha)$, where $S \subseteq X$ is
a smooth hyperplane section, and $\alpha \in H^{d-1}(S, \ZZ) \van$. Then
$T \van$ is an analytic covering space of $\Psm$, with countably many sheets and
countably many connected components. If we let $T \van (\alpha)$ be the component
containing the point $(S_0, \alpha)$, it is easy to see that
\[
	\pi_1 \bigl( T \van (\alpha), (S_0, \alpha) \bigr) =
		\menge{g \in G}{g \cdot \alpha = \alpha}.
\]

Now let $\omega \in H^d(X, \QQ) \prim$ be a nonzero primitive cohomology class. We can use
the tube mapping to construct from $\omega$ a first cohomology class on $T \van (\alpha)$.
Indeed, the rule
\[
	\menge{g \in G}{g \cdot \alpha = \alpha} \to \QQ, \qquad
		g \mapsto \int_{\tau_g(\alpha)} \omega,
\]
defines a homomorphism from the fundamental group of $T \van (\alpha)$ to $\QQ$. By
virtue of Hurewicz' theorem, it corresponds to a class in $H^1 \bigl( T \van (\alpha), \QQ
\bigr)$. It is not hard to show that this class is independent of the choice of base
point on $T \van (\alpha)$. Thus we have a well-defined map
\[
	F \colon H^d(X, \QQ) \prim \to H^1 \bigl( T \van, \QQ \bigr).
\]
Theorem~\ref{thm:tube-main} is the assertion that this map is injective.
In other words, as predicted by Clemens, the topology of the complex manifold $T \van$ is
sufficiently complicated to detect primitive cohomology classes on $X$.

The result has an interesting consequence for the study of Hodge loci on Calabi-Yau
threefolds. Let $X$ be a Calabi-Yau threefold, and let $\omega \in H^0 \bigl( X,
\Omega_X^3 \bigr)$ be a nowhere vanishing holomorphic three-form. Clemens
\cite{Clemens-CY} has shown that the locus of Hodge classes
\[
	\Alg \bigl( T \van \bigr) = 
		\menge{(S, \alpha) \in T \van}
			{\text{$\alpha \in H^{1,1}(S) \cap H^2(S, \ZZ) \van$}}
\]
is the zero locus of a closed holomorphic $1$-form $\Omega$ on $T$, constructed from
$\omega$.  Local integrals of $\Omega$ are referred to as \define{potential functions}
in \cite{Clemens-CY}; as in many other situations, the points in $T$ corresponding to
geometric objects (curves on $X$) are therefore given as the critical locus of these
potential functions.

From the construction in \cite{Clemens-CY}*{p.~735}, it is easy to see that 
\[
	\lbrack \Omega \rbrack = F(\omega) \in H^1 \bigl( T \van, \CC \bigr).
\]
Since the map $F$ is injective by Theorem~\ref{thm:tube-main}, it follows that the
$1$-form $\Omega$ is not exact. In particular, there is no globally defined potential
function on all of $T \van$.

\section{Proof of the main theorem}
\label{sec:proof}

We now describe the proof of Theorem~\ref{thm:tube-main}, referring to later sections
for some of the details.

\subsection*{Dual formulation}

Generally speaking, it is easier to prove that a map is injective than to prove that
it is surjective. With this in mind, we consider the mapping dual to
\eqref{eq:tube-def}. Under the intersection pairing
on $S_0$, the space $V_{\QQ} = H^{d-1}(S_0, \QQ) \van$ is self-dual. Since the
intersection pairing is $G$-invariant, we then have
\[
	\Hom_{\QQ} \bigl( \ker (g - \id), \QQ \bigr) 
	\simeq \coker (g - \id)
	\simeq V_{\QQ} / (g - \id) V_{\QQ}
\]
for any element $g \in G$ acting on $V_{\QQ}$. Similarly, we get
\[
	\Hom_{\QQ} \bigl( H^d(X, \QQ) \prim, \QQ \bigr) \simeq
		H^d(X, \QQ) \prim
\]
by using the intersection pairing on $X$. The dual of the tube mapping is therefore 
\begin{equation} \label{eq:tube-dual}
	H^d(X, \QQ) \prim \to 
		\prod_{g \in G} V_{\QQ} / (g - \id) V_{\QQ}.
\end{equation}
To prove that the tube mapping is surjective, it suffices to show that
\eqref{eq:tube-dual} is injective.

The main advantage to this point of view is that the map \eqref{eq:tube-dual} can
naturally be factored into three simpler maps. We now discuss each of the three in
turn.

\subsection*{The first map}

The first step is to look at the topology of the family of all smooth hyperplane
sections $\pism \colon \famXsm \to \Psm$. From the projection $\famXsm \to X$, we
have a pullback map $H^d(X, \QQ) \to H^d \bigl( \famXsm, \QQ)$.
Now consider the Leray spectral sequence for $\pism$, whose $E_2$-page is
\[
	E_2^{p,q} = H^p \bigl( \Psm, R^q \pisml \QQ \bigr) \Longrightarrow
		H^{p+q} \bigl( \Psm, \QQ \bigr).
\]
Here $R^q \pisml \QQ$ is the local system on $\Psm$ with fiber $H^q(S_0, \QQ)$. The
spectral sequence degenerates at $E_2$ by Deligne's theorem \cite{Voisin}*{p.~379},
because $\pism$ is smooth and projective. Letting $L^{\bullet} H^d \bigl(
\famXsm, \QQ \bigr)$ be the induced filtration on the cohomology of $\famXsm$, we see
in particular that
\[
	H^d \bigl( \famXsm, \QQ \bigr) / L^1 H^d \bigl( \famXsm, \QQ \bigr) \simeq
		E_2^{0,d} = H^0 \bigl( \Psm, R^d \pisml \QQ \bigr)
\]
and
\[
	L^1 H^d \bigl( \famXsm, \QQ \bigr) / L^2 H^d \bigl( \famXsm, \QQ \bigr)
		\simeq E_2^{1,d} = H^1 \bigl( \Psm, R^{d-1} \pisml \QQ \bigr).
\]
By definition, primitive cohomology classes on $X$ restrict to zero on every fiber of
$\pism$, and therefore go to zero in $E_2^{0,d}$. This means that $H^d(X, \QQ) \prim$
is mapped into $L^1 H^d \bigl( \famXsm, \QQ \bigr)$. Composing with the projection to
$E_2^{1,d}$, we obtain a map
\[
	H^d(X, \QQ) \prim \to H^1 \bigl( \Psm, R^{d-1} \pisml \QQ \bigr).
\]
From the decomposition in \eqref{eq:decompS}, we have $R^{d-1} \pisml \QQ =
H^{d-1}(X, \QQ) \oplus ( R^{d-1} \pisml \QQ ) \van$, where the first summand is constant.
Noting that $H^1 \bigl( \Psm, \QQ \bigr) = 0$ because $P \setminus \Psm$ is
irreducible, we find that
\[
	H^1 \bigl( \Psm, R^{d-1} \pisml \QQ \bigr) \simeq 
		H^1 \bigl( \Psm, (R^{d-1} \pisml \QQ) \van \bigr),
\]
and so we obtain the first map in its final form as
\begin{equation} \label{eq:step1}
	H^d(X, \QQ) \prim \to H^1 \bigl( \Psm, (R^{d-1} \pisml \QQ) \van \bigr).
\end{equation}

We shall prove in Section~\ref{sec:Lefschetz} that \eqref{eq:step1} is injective,
provided $H^{d-1}(S_0, \QQ) \van \neq 0$. This is a simple consequence of the
topology of Lefschetz pencils on $X$. 

\subsection*{The second map}

The second step is to represent the cohomology of the local system $(R^{d-1} \pisml
\QQ) \van$ by group cohomology. The group in question is, of course, the fundamental group
$G = \pi_1 \bigl( \Psm, H_0 \bigr)$, which acts on $H^{d-1}(S_0, \QQ) \van$ through monodromy.

In general, given a group $G$ and a $G$-module $M$, the $i$-th group cohomology is
defined as
\[
	H^i(G, M) = \Ext_{\ZZ G}^i (\ZZ, M)
\]
in the category of $\ZZ G$-modules \cite{Weibel}*{p.~161}. In particular, $H^0(G, M)
= M^G$ is the submodule of $G$-invariant elements. The first cohomology $H^1(G, M)$,
which is all we shall use, can be described explicitly as a quotient $Z^1(G, M)
/ B^1(G, M)$, where
\[
	Z^1(G, M) = \menge{\phi \colon G \to M}{\text{$\phi(gh) = g \cdot \phi(h) + \phi(g)$
		for all $g, h \in G$}}
\]
is the group of $1$-cocyles, and
\[
	B^1(G, M) = \menge{\phi \colon G \to M}{\text{there is $x \in M$ with 
		$\phi(g) = g \cdot x - x$ for all $g \in G$}}
\]
the group of $1$-coboundaries for $M$.

There is a well-known correspondence between local systems and representations of the
fundamental group \cite{Voisin}*{Corollaire~15.10 on p.~339}. The following lemma
explains the relationship between the cohomology of the local system and the group
cohomology of the representation.

\begin{lemma} \label{lem:coh-locsys}
Let $\locsys$ be a local system on a connected topological space $B$. Let $M$
be its fiber at some point $b_0 \in B$; it is the representation of the fundamental
group $G = \pi_1(B, b_0)$ corresponding to $\locsys$. Assume that $B$ has a universal
covering space $\tilde{B} \to B$. Then there is a convergent spectral sequence
\[
	E_2^{p,q} = \Ext_{\ZZ G}^p \bigl( H_q(\tilde{B}, \ZZ), M \bigr) 
		\Longrightarrow H^{p+q}(B, \locsys).
\]
In particular, we have $H^1(B, \locsys) \simeq H^1(G, M)$.
\end{lemma}

\begin{proof}
We sketch the proof. Let $S_{\bullet}(\tilde{B}, \ZZ)$ be the
singular chain complex of $\tilde{B}$;
it is a complex of $G$-modules, because $G$ acts on $\tilde{B}$ by deck
transformations.
According to one definition, the cohomology of the local system $\locsys$ is the
cohomology of the complex $\Hom_{\ZZ G} \bigl( S_{\bullet}(\tilde{B}, \ZZ), M
\bigr)$. The spectral sequence in question comes from the double complex
$\Hom_{\ZZ G} \bigl( S_{\bullet}(\tilde{B}, \ZZ), I^{\bullet} \bigr)$,
where $I^{\bullet}$ is any injective resolution of $M$ in the category of $\ZZ
G$-modules.
The second assertion follows immediately from the spectral sequence, because
$H_0(\tilde{B}, \ZZ) \simeq \ZZ$, while $H_1(\tilde{B}, \ZZ) = 0$.
\end{proof}

In our case, the local system $(R^{d-1} \pisml \QQ) \van$ corresponds to the $G$-module
$V_{\QQ} = H^{d-1}(S_0, \QQ) \van$. Thus Lemma~\ref{lem:coh-locsys} gives us an isomorphism
\begin{equation} \label{eq:step2}
	H^1 \bigl( \Psm, (R^{d-1} \pisml \QQ) \van \bigr) \simeq 
		H^1 \bigl( G, V_{\QQ} \bigr)
\end{equation}
with the first group cohomology of $V_{\QQ}$.

\subsection*{The third map}

The third step is of a purely algebraic nature. Namely, for any $G$-module $M$, 
we have a restriction map
\[
	H^1(G, M) \to \prod_{g \in G} H^1 \bigl( \gZZ, M \bigr),
\]
where $\gZZ$ is the cyclic subgroup generated by $g$. From the explicit
description, it is easy to see that $H^1 \bigl( \gZZ, M \bigr) \simeq M / (g -
\id) M$. The resulting map
\begin{equation} \label{eq:restr-M}
	H^1(G, M) \to \prod_{g \in G} M / (g - \id) M
\end{equation}
takes the class of a $1$-cocycle $\phi$ to the element of the product with components $\phi(g) +
(g - \id) M$.

In the case at hand, where $G$ is the fundamental group of $\Psm$, the $G$-module is
$V_{\QQ}$, and the restriction map becomes
\begin{equation} \label{eq:step3}
	H^1 \bigl( G, V_{\QQ} \bigr) \to 
		\prod_{g \in G} V_{\QQ} / (g - \id) V_{\QQ}.
\end{equation}
Unfortunately, \eqref{eq:restr-M} fails to be injective for general $M$ (an example
is given on p.~\pageref{ex:non-injective}); nevertheless, we shall prove its
injectivity for certain $G$-modules, in particular for the vanishing cohomology $V_{\QQ}$.

It should be noted that, as a representation of $G$, the nature of $V_{\QQ}$ is very
different for even and odd values of $d$. This is because the intersection pairing is
symmetric when $d$ is odd, but alternating when $d$ is even. Consequently, the proof
that \eqref{eq:step3} is injective has to be different in the two cases. When $d$ is
odd, it is a straightforward calculation, given in Section~\ref{sec:group-odd}. When
$d$ is even, we show that $H^{d-1}(S_0, \ZZ) \van$, modulo torsion, is a vanishing
lattice \cite{Janssen}. We can then use results by W.~Janssen about the structure of
vanishing lattices to prove the injectivity. Details can be found in
Section~\ref{sec:group-even}.

\subsection*{Conclusion of the proof}

Composing the three maps in \eqref{eq:step1}, \eqref{eq:step2}, and \eqref{eq:step3},
we finally obtain an injective map
\begin{equation} \label{eq:map-main}
	H^d(X, \QQ) \prim \to 
		\prod_{g \in G} V_{\QQ} / (g - \id) V_{\QQ}.
\end{equation}
We shall complete the proof of Theorem~\ref{thm:tube-main} by showing that this map 
agrees with the one in \eqref{eq:tube-dual}. As pointed out above, the injectivity of
\eqref{eq:tube-dual} is equivalent to the surjectivity of the tube mapping, and so we
get our result. \qed

\section{Topology of the universal hypersurface}
\label{sec:Lefschetz}

The main purpose of this section is to show that the map in \eqref{eq:step1} is
injective, as long as $H^{d-1}(S_0, \QQ) \van \neq 0$. As we have seen, this is
the same as showing the injectivity of the map
\[
	H^d(X, \QQ) \prim \to H^1 \bigl( \Psm, R^{d-1} \pisml \QQ \bigr),
\]
derived from the Leray spectral sequence.
Along the way, we shall review several results about the vanishing cohomology of
$S_0$ that are obtained by studying Lefschetz pencils on $X$. 
Throughout, we shall assume that $V_{\QQ} = H^{d-1}(S_0, \QQ) \van \neq 0$. 

\subsection*{Lefschetz pencils}

Recall that $P$ is the space of all hyperplanes (in the ambient projective
space), and $\Psm$ the subset of those $H$ for which $X \cap H$ is smooth. The
\emph{dual variety} $\dualX = P \setminus \Psm$ is the set of hyperplanes such that
$X \cap H$ is singular.  It is an irreducible subvariety of $P$; since we are
assuming that the vanishing cohomology is nontrivial, it is actually a hypersurface,
whose smooth points correspond to hyperplane sections of $X$ with a single ordinary
double point.
\begin{lemma} \label{lem:dualX}
If $V_{\QQ} \neq 0$, then $\dualX$ is a hypersurface in $P$.
\end{lemma}
\begin{proof}
We will prove the converse: if $\dualX$ is not a hypersurface, then necessarily
$V_{\QQ} = 0$. So let us suppose that the codimension of $\dualX$ is at least two.
Choose a line $\PPn{1} \subseteq P$ that does not meet $\dualX$, and let $f \colon
\tilde{X} \to \PPn{1}$ be the restriction of the family of hyperplane sections to
$\PPn{1}$. Then $f$ is smooth and projective, and since $\PPn{1}$
is simply connected, all the local systems $R^q \fl \QQ$ are constant, with fiber 
$H^q(S_0, \QQ)$. Now consider the Leray spectral sequence for the map $f \colon \tilde{X} \to
\PPn{1}$. By Deligne's theorem, it degenerates at $E_2$, and thus gives a short exact
sequence
\begin{diagram}
0 &\rTo& H^2 \bigl( \PPn{1}, R^{d-3} \fl \QQ \bigr)  &\rTo& 
		H^{d-1} \bigl( \tilde{X}, \QQ \bigr) &\rTo& 
		H^0 \bigl( \PPn{1}, R^{d-1} \fl \QQ \bigr) &\rTo& 0;
\end{diagram}
using that $R^q \fl \QQ$ is constant, this amounts to the exactness of the first row
in the following diagram. (All cohomology groups are with coefficients in $\QQ$.)
\begin{diagram}
0 &\rTo& H^2 \bigl( \PPn{1} \bigr) \tensor H^{d-3}(S_0)  &\rTo& 
		H^{d-1} \bigl( \tilde{X} \bigr) &\rTo& 
		H^0 \bigl( \PPn{1} \bigr) \tensor H^{d-1}(S_0) &\rTo& 0 \\
&& \uTo_{\simeq} && \uTo_{\simeq} && \uTo \\
0 &\rTo& H^2 \bigl( \PPn{1} \bigr) \tensor H^{d-3}(X)  &\rTo& 
		H^{d-1} \bigl( \PPn{1} \times X \bigr) &\rTo& 
		H^0 \bigl( \PPn{1} \bigr) \tensor H^{d-1}(X) &\rTo& 0
\end{diagram}

The two vertical maps are isomorphisms because of the Hyperplane Theorem. Indeed, we have
already seen that $H^{d-3}(X, \QQ) \simeq H^{d-3}(S_0, \QQ)$.  On the other hand,
$\tilde{X} \subseteq \PPn{1} \times X$ is itself a smooth very ample hypersurface of
dimension $d$, and so we also have $H^{d-1} \bigl( \PPn{1} \times X, \QQ \bigr)
\simeq H^{d-1} \bigl( \tilde{X}, \QQ \bigr)$. It follows that $H^{d-1}(X, \QQ) \simeq
H^{d-1}(S_0, \QQ)$, which means that $V_{\QQ} = H^{d-1}(S_0, \QQ) \van$ is reduced to zero.
\end{proof}

Now take any Lefschetz pencil of hyperplane sections of $X$ containing
$S_0$; in other words, a line $\PPn{1} \subseteq P$ through the base point $H_0 \in P$
that meets  $\dualX$ transversely in finitely many points.  Also let $B = \PPn{1}
\cap \Psm$ be the smooth locus of the pencil, and let $0 \in B$ be the point
whose image is $H_0$.
We write $\tilde{X} \to \PPn{1}$ for the restriction of the family $\famX \to P$ to the line,
and $U \subseteq \tilde{X}$ for the part that lies over $B$. The following diagram shows
the relevant maps; all diagonal arrows are inclusions of open subsets.
\begin{diagram}[size=1.4em,height=1.3em,tight]
& & \tilde{X} & & \rTo & & \famX & & \rTo^q& & X \\
& \ruTo & \vLine & & & \ruTo & \\
U & \rTo & \HonV & & \famXsm & & \dTo_{\pi} \\
& & \dTo & & \dTo & & \\
\dTo^f & & \PPn{1} & \hLine & \VonH & \rTo & P \\
& \ruTo & & & & \ruTo & \\
B & & \rTo &  & \Psm & &  \\
\end{diagram}
% The other cube
% \begin{diagram}
% U & & \rTo & & \famXsm & & \\
% & \rdTo & & & \vLine & \rdTo & \\
% \dTo & & \tilde{X} & \rTo & \HonV & & \famX \\
% & & \dTo & & \dTo & & \\
% B & \hLine & \VonH & \rTo & \Psm & & \dTo_{\pi} \\
% & \rdTo & & & & \rdTo & \\
% & & \PPn{1} & & \rTo & & P \\
% \end{diagram}

We know from Lemma~\ref{lem:dualX} that $D = \PPn{1} \cap \dualX$ is nonempty; say $D
= \{ t_1, \dotsc, t_n \}$, with all $t_i$ distinct and different from the base point
$0$. Let $S_i$ be the hyperplane section of $X$ corresponding to the point $t_i$;
each $S_i$ has a single ordinary double point (ODP). From a local analysis around an ODP
singularity, it is known that $S_0$ contains an embedded $(d-1)$-sphere for each $i$,
the so-called \define{vanishing cycle} for the singularity on $S_i$. Moreover, the
homotopy type of $S_i$ is that of $S_0$ with a $d$-cell attached along the vanishing
cycle \cite{Voisin}*{p.~322}.

The vanishing homology $\ker \bigl( \il \colon H_{d-1}(S_0, \ZZ) \to H_{d-1}(X, \ZZ) \bigr)$
is generated over $\ZZ$ by the classes of these spheres \cite{Voisin}*{Lemme~14.26 on
p.~327}. Writing $e_i$ for the cohomology class Poincar\'e dual to the $i$-th
vanishing cycle, the $e_i$ thus generate the vanishing cohomology with integer coefficients
\[
	H^{d-1}(S_0, \ZZ) \van = 
		\ker \bigl( \il \colon H^{d-1}(S_0, \ZZ) \to H^{d+1}(X, \ZZ) \bigr).
\]
Since $\dualX$ is irreducible, it is further known \cite{Voisin}*{Corollaire~15.24 on
p.~353} that all the $e_i$ lie in one
orbit of the monodromy action of $\pi_1(B, 0)$ on $H^{d-1}(S_0, \ZZ)$. In particular,
we have $e_i \neq 0$ in $H^{d-1}(S_0, \QQ) \van$, because we are assuming that
the latter is nontrivial.

\begin{lemma} \label{lem:Si-trivial}
Classes in $H^d(X, \QQ) \prim$ have trivial restriction to each of the singular
hyperplane sections $S_i$.
\end{lemma}
\begin{proof}
The singular hyperplane section $S_i$ is homotopy-equivalent to $S_0$ with a $d$-cell attached along the
$i$-th vanishing cycle. From the Mayer-Vietoris sequence in cohomology, we thus get
an exact sequence isomorphic to
\begin{diagram}
H^{d-1}(S_0, \QQ) &\rTo& H^{d-1} \bigl( \SSn{d-1}, \QQ \bigr) &\rTo& 
		H^d(S_i, \QQ) &\rTo& H^d(S_0, \QQ) &\rTo& 0.
\end{diagram}
Since $e_i \neq 0$, the first map in the sequence is nontrivial, and so
$H^{d-1}(S_i, \QQ) \simeq H^{d-1}(S_0, \QQ)$. In particular, every primitive
cohomology class on $X$ has trivial restriction to $S_i$.
\end{proof}

\begin{lemma} \label{lem:XU}
The pullback map $H^d(X, \QQ) \prim \to H^d(U, \QQ)$ is injective.
\end{lemma}
\begin{proof}
The complement of $U$ in $\tilde{X}$ is the disjoint union of the singular
fibers $S_i$; thus we have an exact sequence
\begin{diagram}
\dotsb &\rTo& H_c^d(U, \QQ) &\rTo& H^d \bigl( \tilde{X}, \QQ \bigr) &\rTo&
	\bigoplus_{i=1}^n H^d(S_i, \QQ) &\rTo& H_c^{d+1}(U, \QQ) &\rTo& \dotsb
\end{diagram}
for cohomology with compact support. As $U$ is a manifold, 
\[
	H_c^d(U, \QQ) \simeq \Hom \bigl( H^d(U, \QQ), \QQ \bigr),
\]
with the isomorphism given by integration over $U$.

Now let $\omega \in H^d(X, \QQ) \prim$ be any class whose pullback $\qu \omega$ has 
trivial restriction to $U$. The functional $H^d \bigl( \tilde{X}, \QQ \bigr) \to \QQ$, given by
integrating against $\qu \omega$, is then zero on $H_c^d(U, \QQ)$, and thus factors
through the image of $H^d \bigl( \tilde{X}, \QQ \bigr) \to \bigoplus_i H^d(S_i,
\QQ)$. Let $\lambda \colon \bigoplus_i H^d(S_i, \QQ) \to \QQ $ be any
extension to the entire direct sum. For each $\alpha \in H^d(X, \QQ) \prim$
we then have
\[
	\int_X \omega \cup \alpha = 
	\int_{\tilde{X}} \qu \omega \cup \qu \alpha = 
		\lambda \Bigl( \restr{\alpha}{S_1}, \dotsc, \restr{\alpha}{S_n} \Bigr) = 0
\]
by Lemma~\ref{lem:Si-trivial}. But the intersection pairing on $H^d(X, \QQ) \prim$ is
nondegenerate, and so $\omega = 0$.
\end{proof}

\subsection*{Injectivity of the map}

Now consider the Leray spectral sequence for the map $f \colon U \to B$. Since $B$
has the homotopy-type of a bouquet of circles, the spectral sequence degenerates, and
we get a short exact sequence
\begin{diagram}
0 &\rTo& H^1 \bigl( B, R^{d-1} \fl \QQ \bigr) &\rTo& 
		H^d(U, \QQ) &\rTo& H^0 \bigl( B, R^d \fl \QQ \bigr) &\rTo& 0.
\end{diagram}
By definition, classes in $H^d(X, \QQ) \prim$ go to zero in the group on the right;
Lemma~\ref{lem:XU} then lets us conclude that the induced map
\[
	H^d(X, \QQ) \prim \to H^1 \bigl( B, R^{d-1} \fl \QQ \bigr)
\]
has to be injective. This immediately implies the injectivity of \eqref{eq:step1}.
To see this, note that functoriality of the Leray spectral sequence gives a
factorization
\[
	H^d(X, \QQ) \prim \to H^1 \bigl( \Psm, R^{d-1} \pisml \QQ \bigr) \to
		H^1 \bigl( B, R^{d-1} \fl \QQ \bigr).
\]
The right-hand map is injective, because the map of fundamental groups
\[
	\pi_1(B, 0) \to \pi_1 \bigl( \Psm, H_0 \bigr) = G
\]
is surjective by Zariski's theorem \cite{Voisin}*{Th\'eor\`eme~15.22 on p.~351}. Thus
the left-hand map also injective, proving our claim.

\subsection*{Vanishing cycles with intersection number one}

We have seen that the vanishing cohomology $H^{d-1}(S_0, \ZZ) \van$ is generated by
the Poincar\'e duals $e_i$ of the vanishing cycles for any Lefschetz pencil.
More generally, we shall refer to any element in the orbit $\Delta = G \cdot \{ e_1,
\dotsc, e_n \}$ as a vanishing cycle. As shown above, all $\delta \in \Delta$ are
nontrivial even as elements of $V_{\QQ} = H^{d-1}(S_0, \QQ) \van$.

The fundamental group $\pi_1(B, 0)$ is isomorphic to a free group on $(n-1)$
letters; in fact, a set of generators is given by taking, for each $i = 1, \dotsc,
n$, a loop $g_i$ based at $0$ that goes exactly once around the point $t_i$ with
positive orientation, but not around any of the other $t_j$. The only relation is the
obvious one, namely that $g_1 \dotsc g_n = 0$. By Zariski's theorem, $G$ itself is
also generated by the $g_i$.

The monodromy action of each $g_i$ on $H^{d-1}(S_0, \ZZ)$ is described explicitly by
the Picard-Lefschetz formula \cite{Voisin}*{Th\'eor\`eme~15.16 on p.~345}
\begin{equation} \label{eq:Picard-Lefschetz}
	g_i \cdot \alpha = \alpha - \eps_d (\alpha, e_i) e_i,
\end{equation}
where $(\argbl, \argbl)$ is the intersection pairing on $S_0$, and $\eps_d =
(-1)^{d(d-1)/2}$. This has different consequences for odd and even values of $d$:
\begin{enumerate}[label=(\roman{*}), ref=(\roman{*})]
\item When $d$ is odd, $S_0$ has even dimension, and the intersection pairing is
symmetric. Moreover, each vanishing cycle has self-intersection number $2 \eps_d$,
and $g_i^2$ acts trivially on $H^{d-1}(S_0, \ZZ)$.
\item When $d$ is even, $S_0$ has odd dimension, and the intersection pairing is
skew-symmetric. Consequently, the self-intersection of $e_i$ is zero, and the element
$g_i$ is of infinite order.
\end{enumerate}
The same formulas are of course true for every vanishing cycle $\delta \in \Delta$.

To analyze the structure of $V_{\QQ}$ for even values of $d$, we will need the
following lemma about the set $\Delta$. It is the main step in showing that
$H^{d-1}(S_0, \ZZ) \van$ is a skew-symmetric vanishing lattice \cite{Janssen}.
\begin{lemma} \label{lem:number-one}
Assume that $d = \dim X$ is even. Then there are two vanishing cycles
$\delta_1, \delta_2 \in \Delta$ with $(\delta_1, \delta_2) = 1$.
\end{lemma}

\begin{proof}
As observed in \cite{Janssen}*{p.~132}, it suffices to show that there is a singular
hyperplane section $S' \subseteq X$ with an isolated singularity that is not an
ordinary double point. Indeed, the vanishing homology of the Milnor fiber $F$ of such a
singularity embeds into $H_{d-1}(S_0, \ZZ) \van$ by \cite{Beauville}*{p.~9}, in such
a way that vanishing cycles map to vanishing cycles. The Milnor fiber has the
homotopy type of a bouquet of $(d-1)$-spheres; the number of spheres is the Milnor
number $\mu$ of the singular point. If the singularity is not an ordinary double
point, then $\mu \geq 2$, and so there are (at least) two independent vanishing cycles
on $F$ with intersection number one. We can then take $\delta_1$ and $\delta_2$ to be
their images in $H_{d-1}(S_0, \ZZ) \van$.

To find such a hyperplane section $S'$, let $\PPn{2} \subseteq P$ be a general plane
containing the base point, and $C = \PPn{2} \cap \dualX$. The curve $C$ is
irreducible, and its only singularities are nodes and cusps. A node of $C$
corresponds to a hyperplane section of $X$ with two ordinary double points; a cusp
corresponds to a hyperplane section with one isolated singularity of Milnor number
two. To prove the lemma, it is therefore enough to show that $C$ has at least one cusp.

By the more precise version of Zariski's theorem, the fundamental group of $\PPn{2}
\setminus C$ is isomorphic to $G$. If $C$ had only nodes and no cusps, then this
group would be Abelian \cite{Deligne}, and hence a finite cyclic group, since $C$ is
irreducible. In particular, the action of each vanishing cycle would be of finite
order. Since $d$ is even, this possibility is ruled out by our assumption that
$H^{d-1}(S_0, \QQ) \van \neq 0$.  
\end{proof}

\section{An application of Nori's Connectivity Theorem}
\label{sec:Nori}

H.~Clemens observed that one can obtain a much stronger result, namely that
\eqref{eq:step1} is an \emph{isomorphism}, from the Connectivity Theorem of M.~Nori
\cite{Nori}. For this to be true, it is necessary to assume that $X$ is of
sufficiently large degree in the ambient projective space. We give the proof of
Clemens' observation in this section.

As before, we let $d = \dim X$ be the dimension of the smooth projective variety
$X$. We also continue to write $\pism \colon \famXsm \to
\Psm$ for the family of all smooth hyperplane sections of $X$. 
Nori's Connectivity Theorem \cite{Voisin}*{Section~20.1} is the statement that the restriction map
\begin{equation} \label{eq:nori}
	H^k \bigl( \Psm \times X, \QQ \bigr) \to H^k \bigl( \famXsm, \QQ \bigr)
\end{equation}
is an isomorphism for all $k \leq 2d-3$, and injective for $k = 2d-2$, provided the
embedding of $X$ is of sufficiently high degree, which we assume from now on.

The following result is an useful consequence of Nori's theorem.

\begin{proposition}[Clemens] \label{prop:Nori}
The map
\[
	H^d(X, \QQ) \prim \to H^1 \bigl( \Psm, (R^{d-1} \pisml \QQ) \van \bigr),
\]
induced by the Leray spectral sequence for $\pism$, 
is an isomorphism if $d \geq 3$, and at least injective if $d = 2$.
\end{proposition}

\begin{proof}
We consider the Leray spectral sequence for the smooth projective map $\pism$, whose
$E_2$-page is given by
\[
	E_2^{p,q} = H^p \bigl( \Psm, R^q \pisml \QQ \bigr) 
		\Longrightarrow H^{p+q} \bigl( \famXsm, \QQ \bigr).
\]
Writing $L^{\bullet}$ for the induced filtration on the cohomology of
$\famXsm$, we have for each $p \geq 0$ a short exact sequence
\begin{diagram}[l>=2em]
	L^{p+1} H^d \bigl( \famXsm, \QQ \bigr) &\rIntoBold& 
		L^p H^d \bigl( \famXsm, \QQ \bigr) &\rOnto&
		H^p \bigl( \Psm, R^{d-p} \pisml \QQ \bigr),
\end{diagram}
because the spectral sequence degenerates at the $E_2$-page by a theorem of Deligne's
\cite{Voisin}*{p.~379}.
A similar result is true for the projection map $\Psm \times X \to \Psm$;
since the Leray spectral sequence is functorial, we get for all $p \geq 0$ a
commutative diagram
\begin{diagram}[midshaft,l>=2em]
	L^{p+1} H^d \bigl( \Psm \times X \bigr) 
		&\rIntoBold& L^p H^d \bigl( \Psm \times X \bigr) &\rOnto&
		H^p \bigl( \Psm \bigr) \times H^{d-p}(X) \\
	 \dTo^{r_{p+1}} & & \dTo^{r_p} & & \dTo^{q_p}  \\
	L^{p+1} H^d \bigl( \famXsm \bigr) &\rIntoBold& L^p H^d \bigl( \famXsm \bigr) &\rOnto& 
		H^p \bigl( \Psm, R^{d-p} \pisml \QQ \bigr)
\end{diagram}
with exact rows. (All cohomology groups are with $\QQ$-coefficients.)

We now analyze those diagrams. For $p \geq 2$, the map $H^{d-p}(X, \QQ) \to
H^{d-p}(S, \QQ)$ is an isomorphism for every smooth hyperplane section $S = X \cap H$
by the Hyperplane Theorem; thus $q_p$ is an isomorphism for $p \geq 2$. An easy
induction, combined with the Five Lemma, shows that $r_p$ is therefore an isomorphism
for $p \geq 2$ as well.

For $p = 1$, we have
\begin{diagram}[midshaft,l>=2em]
	L^2 H^d \bigl( \Psm \times X \bigr) 
		&\rIntoBold& L^1 H^d \bigl( \Psm \times X \bigr) &\rOnto&
		H^1 \bigl( \Psm \bigr) \times H^{d-1}(X) \\
	 \dTo^{\simeq} & & \dTo^{r_1} & & \dTo^{q_1}  \\
	L^2 H^d \bigl( \famXsm \bigr) &\rIntoBold& L^1 H^d \bigl( \famXsm \bigr) &\rOnto& 
		H^1 \bigl( \Psm, R^{d-1} \pisml \QQ \bigr)
\end{diagram}
and since $R^{d-1} \pisml \QQ = H^{d-1}(X, \QQ) \oplus (R^{d-1} \pisml \QQ) \van$, we
obtain an isomorphism
\[
	\coker r_1 \simeq \coker q_1 \simeq 
		H^1 \bigl( \Psm, (R^{d-1} \pisml \QQ) \van \bigr)
\]

Finally, we consider the diagram for $p = 0$, where it becomes
\begin{diagram}[midshaft,l>=2em]
	L^1 H^d \bigl( \Psm \times X \bigr) 
		&\rIntoBold& H^d \bigl( \Psm \times X \bigr) &\rOnto&
		H^0 \bigl( \Psm \bigr) \times H^d(X) \\
	 \dTo^{r_1} & & \dTo^{r_0} & & \dTo^{q_0}  \\
	L^1 H^d \bigl( \famXsm \bigr) &\rIntoBold& H^d \bigl( \famXsm \bigr) &\rOnto& 
		H^0 \bigl( \Psm, R^d \pisml \QQ \bigr).
\end{diagram}
By Nori's theorem, the map $r_0$ is an isomorphism if $2d-3 \geq d$, equivalently
$d \geq 3$, and at least injective if $d=2$. Thus the map $\ker q_0 \to \coker r_1$,
given by the Snake Lemma, is an isomorphism for $d \geq 3$, and injective for $d =
2$. Since we clearly have $\ker q_0 = H^d(X, \QQ) \prim$, the assertion follows.
\end{proof}

\begin{remark}
When combined with the Lefschetz theorems and the isomorphism in \eqref{eq:step2},
Proposition~\ref{prop:Nori} implies that the cohomology
$H^{\ast}(X, \QQ)$ can be reconstructed from the following data (for $d \geq 2$):
\begin{enumerate}[label=(\roman{*}), ref=(\roman{*})]
\item The cohomology ring $H^{\ast}(S, \QQ)$ of a smooth hypersurface section
$S = X \cap H$ of sufficiently high degree.
\label{it:Nori-1}
\item The cohomology class $\iu \lbrack S \rbrack \in H^2(S, \QQ)$.
\label{it:Nori-2}
\item The representation of the fundamental group $G$ of $\Psm$ on 
the vector space $V_{\QQ} = H^{d-1}(S, \QQ) \van$, more precisely its first group cohomology 
$H^1(G, V)$.  
\label{it:Nori-3}
\end{enumerate}

To see this, note that the class in \ref{it:Nori-2} determines the decomposition
of $H^{d-1}(S, \QQ)$ into vanishing cohomology and $H^{d-1}(X, \QQ)$. Thus the entire
cohomology ring of $X$ can be recovered from \ref{it:Nori-1} and \ref{it:Nori-2} ,
with the exception of the group $H^d(X, \QQ) \prim$.  But according to
Proposition~\ref{prop:Nori}, we have
\[
	H^d(X, \QQ) \prim \simeq H^1 \bigl( \Psm, (R^{d-1} \pisml \QQ) \van \bigr)
		\simeq H^1 \bigl( G, V_{\QQ} \bigr),
\]
where the second isomorphism is the one in \eqref{eq:step2}. It would be interesting
to have a description of the ring structure on $H^{\ast}(X, \QQ)$ in this setting.
\end{remark}

\section{Detecting group cohomology classes: the odd case}
\label{sec:group-odd}

In this section, we show that the restriction map \eqref{eq:step3} is injective when
$d$ is odd. As it happens, this can be proved by using very little of the structure
of the vanishing cohomology, and so we shall treat the problem abstractly first.

\subsection*{Injectivity of the restriction map}

We consider a finite-dimensional $\QQ$-vector space $V_{\QQ}$ with a symmetric bilinear form $B
\colon V_{\QQ} \tensor V_{\QQ} \to \QQ$, and a finitely generated group $G$ acting on $V_{\QQ}$,
subject to the following two assumptions:
\begin{enumerate}
\item There are distinguished elements $e_1, \dotsc, e_n \in V_{\QQ}$ with $\pair{e_i}{e_i}
= 2$.
\item There are generators $g_1, \dotsc, g_n$ for $G$, such that
\[
	g_i \cdot v = v - \pair{v}{e_i} e_i
\]
for all $v \in V_{\QQ}$. 
\end{enumerate}
It follows that the action of $G$ preserves the bilinear form, and that each $g_i^2$
acts trivially. In this situation, the restriction map is injective.

\begin{proposition} \label{prop:group-odd}
Let $V_{\QQ}$ be a finite-dimensional $\QQ$-vector space with an action by a group $G$,
subject to the assumptions just stated. Then the restriction map
\[
	H^1 \bigl( G, V_{\QQ} \bigr) \to \prod_{g \in G} V_{\QQ} / (g - \id) V_{\QQ}
\]
is injective.
\end{proposition}

\begin{proof}
Let $\phi \in Z^1 \bigl( G, V_{\QQ} \bigr)$ represent an arbitrary class in the kernel of the
restriction map. This means that for every $g \in G$, there is some $v \in V_{\QQ}$ with
the property that $\phi(g) = gv - v$.  Of course, $v$ is allowed to depend on $g$.
To prove the asserted injectivity, we need to show that $\phi \in B^1 \bigl( G, V_{\QQ} \bigr)$.

We shall do this in two steps. Re-indexing the generators $g_1, \dotsc, g_n$ of $G$,
if necessary, we may assume that the vectors $e_1, \dotsc, e_p$ are linearly
independent, while $e_{p+1}, \dotsc, e_n$ are linearly dependent on $e_1, \dotsc,
e_p$. The \emph{first step} is to show that we can subtract from $\phi$ a suitable element of
$B^1 \bigl( G, V_{\QQ} \bigr)$ to get $\phi(g_i) = 0$ for $i = 1, \dotsc, p$.

By assumption, there is a vector $v \in V_{\QQ}$ such that
\[
	\phi(g_p \dotsm g_1) = g_p \dotsm g_1 v - v;
\]
after subtracting from $\phi$ the element $(g \mapsto gv - v) \in B^1(G, V_{\QQ})$, we have
$\phi(g_p \dotsm g_1) = 0$. Furthermore, for each $i = 1, \dotsc, p$, there is some
$v_i \in V_{\QQ}$ with
\[
	\phi(g_i) = g_i v_i - v_i = - \pair{v_i}{e_i} e_i = a_i e_i,
\]
where $a_i = - \pair{v_i}{e_i} \in \QQ$. According to Lemma~\ref{lem:expand} below,
we can write
\[
	\phi(g_p \dotsm g_1) = \sum_{k=1}^p b_k e_k,
\]
with coefficients $b_k$ that satisfy the recursive relations given in the lemma.
But $e_1, \dotsc, e_p$ are linearly independent, and therefore $b_1 = \dotsb = b_p =
0$. The relations imply that $a_1 = \dotsb = a_p = 0$, and so we obtain $\phi(g_i) = 0$ for $i = 1,
\dotsc, p$.

In the \emph{second step}, we show that $\phi$ is now actually zero.
For this, we only need to prove that $\phi(g_i) = 0$ for $i = p+1, \dotsc, n$,
because all the $g_i$ together generate $G$ and $\phi$ is a cocycle. By symmetry, it
obviously suffices to consider just $g_{p+1}$. Since $e_{p+1}$ is linearly dependent
on $e_1, \dotsc, e_p$, we can write
\[
	e_{p+1} = \sum_{i=1}^p c_i e_i
\]
for certain coefficients $c_i \in \QQ$; these are subject to the condition that
\[
	2 = \pair{e_{p+1}}{e_{p+1}} = \sum_{i,j=1}^p c_i \pair{e_i}{e_j} c_j.
\]
If we let $c$ be the column vector with coordinates $c_i$, and $E$ the symmetric $p \times
p$-matrix with entries $E_{ij} = \pair{e_i}{e_j}$, we can put the condition into
the form
\begin{equation} \label{eq:aEa}
	2 = \transp{c} E c.
\end{equation}

As before, there is a vector $v \in V_{\QQ}$ with $\phi(g_{p+1}) = g_{p+1} v - v = 
-\pair{v}{e_{p+1}} e_{p+1}$, and if we set $\eta = - \pair{v}{e_{p+1}} \in \QQ$, we have
\[
	\phi(g_{p+1}) = \eta \cdot e_{p+1} = \eta \cdot \sum_{j=1}^p c_j e_j.
\]
We may also find $w \in V_{\QQ}$ such that
\[
	\phi(g_{p+1} g_p \dotsm g_1) = g_{p+1} g_p \dotsm g_1 w - w.
\]
Now $\phi(g_i) = 0$ for $i = 1, \dotsc, p$, and so we get $\phi(g_{p+1} g_p \dotsm
g_1) = \phi(g_{p+1})$ from the fact that $\phi$ is a cocycle. Since $g_{p+1} e_{p+1}
= - e_{p+1}$, we calculate that
\begin{align*}
	- \eta \cdot e_{p+1} &= g_{p+1} \cdot \phi(g_{p+1}) 
	= g_{p+1} \cdot \bigl( g_{p+1} g_p \dotsm g_1 w - w  \bigr) \\
		&= \bigl( g_p \dotsm g_1 w - w \bigr) - \bigl( g_{p+1} w - w \bigr)
		= \bigl( g_p \dotsm g_1 w - w \bigr) + \pair{w}{e_{p+1}} e_{p+1}
\end{align*}
Let $x_i = - \pair{w}{e_i}$. An application of Lemma~\ref{lem:expand} to the cocycle
$(g \mapsto g w - w)$ shows that $g_p \dotsm g_1 w - w = \sum y_j e_j$, where
\begin{equation} \label{eq:recursive}
	y_1 = x_1  \qquad \text{and} \qquad  
	y_{k+1} = x_{k+1} - \sum_{i=1}^k E_{i,k+1} y_i.
\end{equation}
From our calculation, we now obtain a linear relation between $e_1, \dotsc, e_p$,
namely
\[
	- \eta \cdot \sum_{j=1}^p c_j e_j = \sum_{j=1}^p y_j e_j - 
		\sum_{i,j=1}^p c_i x_i c_j e_j.
\]
But $e_1, \dotsc, e_p$ are linearly independent, and we deduce that
\[
	\eta \cdot c_j = \sum_{i=1}^p c_i x_i c_j - y_j
\]
for all $j = 1, \dotsc, p$, which we can write as a vector equation
\begin{equation} \label{eq:eta}
	\eta \cdot c = \transp{c} x \cdot c - y.
\end{equation}

The recursive relations in \eqref{eq:recursive} for the $y_j$ can be put into
the form $x = S y$, where $S$ is a lower-triangular matrix with entries
\[
	S_{ij} = 
	\begin{cases} 
		E_{ij} \quad & \text{if $i > j$,} \\
		1      \quad & \text{if $i = j$,} \\
		0		 \quad & \text{if $i < j$.} \\
	\end{cases}
\]
But now $E = S + \transp{S}$, because $E$ is symmetric and its diagonal entries are
all equal to $2$. From \eqref{eq:aEa}, we find that
\[
	2 = \transp{c} E c = \transp{c} S c + \transp{c} \transp{S} c
		= 2 \transp{c} S c,
\]
and so $1 = \transp{c} S c$. Now apply $\transp{c} S$ to the equation in
\eqref{eq:eta}
to get
\[
	\eta = \transp{c} S c \cdot \eta = \transp{c} x \cdot \transp{c} S c 
		- \transp{c} S y = \transp{c} x - \transp{c} S y = 
		\transp{c} (x - Sy) = 0.
\]
This shows that $\phi(g_{p+1}) = \eta \cdot e_{p+1} = 0$, and we have our result.
\end{proof}

The following simple lemma was used twice during the proof. To keep it
general, we do not assume anything about the bilinear form; in this way, it also
applies to the even case in the next section.

\begin{lemma} \label{lem:expand}
Let $V_{\QQ}$ be a $\QQ$-vector space with a bilinear form $B \colon V_{\QQ} \tensor
V_{\QQ} \to \QQ$, and with an action by a group $G$. Assume that there are elements
$g_1, \dotsc, g_n$ of $G$, and vectors $e_1, \dotsc, e_n \in V_{\QQ}$, such that
$g_i v = v - \pair{v}{e_i} e_i$ holds for every $v \in V_{\QQ}$. Let $\phi \in Z^1
\bigl( G, V_{\QQ} \bigr)$ be a $1$-cocycle satisfying $\phi(g_i) = a_i e_i$ for all
$i$. Then we have 
\[
	\phi(g_n \dotsm g_1) = \sum_{k=1}^n b_k e_k,
\]
and the coefficients $b_k \in \QQ$ are determined by the recursive relations
\begin{equation} \label{eq:relations}
	b_1 = a_1 \qquad \text{and} \qquad
		b_{k+1} = a_{k+1} - \sum_{i=1}^k \pair{e_i}{e_{k+1}} b_i.
\end{equation}
\end{lemma}

\begin{proof}
This is easily proved by induction on $n$.
\end{proof}

\begin{example} \label{ex:non-injective}
It should be pointed out that the restriction map \eqref{eq:restr-M} is in general not injective without
some assumptions on the $G$-module $M$. Here is a simple example of
this phenomenon.  Let $G = \ZZ^2$ be the free Abelian group on two generators, acting
on $M = \QQ^3$ by the two commuting matrices
\[
	A_1 = \begin{pmatrix}
		1 & 0 & 1 \\ 0 & 1 & 0 \\ 0 & 0 & 1\\ 
	\end{pmatrix}
	\qquad \text{and} \qquad
	A_2 = \begin{pmatrix}
		1 & 2 & 2 \\ 0 & 1 & 2 \\ 0 & 0 & 1\\ 
	\end{pmatrix}.
\]
Define $\phi \colon G \to M$ by the rule $\phi(a, b) = (a, 0, 0)$. One easily
verifies that $\phi$ gives a non-zero element in $H^1(G, M)$, but that it goes to zero
under the restriction map
\[
	H^1(G, M) \to \prod_{g \in G} M / (g - \id) M.
\]
Thus Proposition~\ref{prop:group-odd} does not remain true for arbitrary representations.
\end{example}

\subsection*{Conclusion of the argument}

Proposition~\ref{prop:group-odd} can now be applied to the vanishing cohomology
$V_{\QQ} = H^{d-1}(S_0, \QQ) \van$ to show that \eqref{eq:step3} is injective when
$d$ is odd.
Indeed, it is clear from the results in Section~\ref{sec:Lefschetz} that
$V_{\QQ}$ satisfies all the assumptions of the proposition, if we set $\pair{u}{v} =
\eps_d (u, v)$. We can take for $e_1, \dotsc, e_n$ the vanishing cycles in an
arbitrary Lefschetz pencil on $X$, and for $g_1, \dotsc, g_n$ the corresponding
generators of the fundamental group. The identity $g_i v = v - \pair{v}{e_i} e_i$ is
then simply the Picard-Lefschetz formula \eqref{eq:Picard-Lefschetz}.
We conclude that \eqref{eq:step3} is injective when $V_{\QQ}$ is the vanishing
cohomology. This completes the proof that \eqref{eq:map-main} is injective when
the dimension of $X$ is odd.

\section{Detecting group cohomology classes: the even case}
\label{sec:group-even}

The purpose of this section is to prove that the restriction map
\begin{equation} \label{eq:inj-even}
	V_{\QQ} \to
		\prod_{g \in G} V_{\QQ} / (g - \id) V_{\QQ}
\end{equation}
is also injective when $d = \dim X$ is even. This is more subtle than in the case of odd
$d$, and we will need to use the fact that the vanishing cohomology and the monodromy
action can all be defined over $\ZZ$. The integral vanishing cohomology $H^{d-1}(S_0,
\ZZ) \van$, modulo torsion, is an example of a skew-symmetric vanishing lattice
\cite{Janssen}. We therefore have to begin by reviewing some results about the structure of
skew-symmetric vanishing lattices, due to W.~Janssen. 

\subsection*{Skew-symmetric vanishing lattices}

Let $V$ be a free $\ZZ$-module of finite rank, with an alternating bilinear form $B
\colon V \tensor V \to \ZZ$. Let $\Sp(V)$ be the group of all automorphisms of $V$
that preserve $B$. For every element $v \in V$, we can define a symplectic
transvection $T_v \in \Sp(V)$ by the formula $T_v(x) = x - \pair{x}{v} v$. 
With the monodromy representation on $H^{d-1}(S_0, \ZZ)$ and the facts in
Section~\ref{sec:Lefschetz} in mind, we are interested
in subgroups of $\Sp(V)$ generated by transvections. Given a subset $\Delta \subseteq
V$, we write $\Gamma_{\Delta}$ for the subgroup of $\Spsh(V)$ generated by all
$T_{\delta}$, for $\delta \in \Delta$.

As a matter of fact, all transvections are contained in a (potentially smaller) group
$\Spsh(V)$, which we now define. The form induces a linear map $j \colon V \to
\Hom(V, \ZZ)$, given by the rule $j(v) = \pair{v}{\argbl}$; in general, it is neither
injective nor surjective without further assumptions on $B$. Now $\Sp(V)$ naturally
acts on the dual module $\Hom (V, \ZZ)$ as well, by setting $(g \lambda)(x) =
\lambda(g^{-1} x)$ for $x \in V$ and $\lambda \in \Hom(V, \ZZ)$, and the map $j$ is
equivariant.  We let $\Spsh(V)$ be the subgroup of those $g \in \Sp(V)$ that act
trivially on $\Hom(V, \ZZ) / j(V)$.  Concretely, this means that
\[
\begin{split}
	\Spsh(V) = \menge{g \in \Sp(V)}{&\text{for any $\lambda \in \Hom(V, \ZZ)$, there
		exists $v \in V$} \\
	&\text{such that $\lambda(gx - x) = \pair{v}{x}$ for all $x \in V$}}.
\end{split}
\]
It is easy to see that $T_v \in \Spsh(V)$; each $\Gamma_{\Delta}$ is
therefore a subgroup of $\Spsh(V)$.

We now come to the main definition. A (skew-symmetric) \define{vanishing lattice} in
$V$ is a subset $\Delta \subseteq V$ with the following three properties:
\begin{enumerate}
\item The set $\Delta$ generates $V$.
\item $\Delta$ is a single orbit under the action of $\Gamma_{\Delta}$.
\item There exist two elements $\delta_1, \delta_2 \in \Delta$ such that
$\pair{\delta_1}{\delta_2} = 1$.
\end{enumerate}
In that case, $\Gamma_{\Delta}$ is called the \define{monodromy group} of the vanishing
lattice. 

W.~Janssen has completely classified vanishing lattices. One of his main technical
results is the following theorem.

\begin{theorem}[\cite{Janssen}*{Theorem~2.5}] \label{thm:Janssen}
Let $\Delta \subseteq V$ be a vanishing lattice. Then the monodromy group
of $\Delta$ contains the congruence subgroup
\[
	\Sptsh(V) = \menge{g \in \Sp(V)}{\text{$g$ acts trivially on $\Hom(V, \ZZ) /
		j(2V)$}}.
\]
In particular, $\Gamma_{\Delta}$ is itself of finite index in $\Spsh(V)$.
\end{theorem}

We shall now use Janssen's theorem to show that $\Gamma_{\Delta}$ contains a finite-index
subgroup with a particularly convenient set of generators. This is crucial
in proving the injectivity of \eqref{eq:inj-even}. 

\begin{lemma} \label{lem:finite-index}
Let $V$ be a free $\ZZ$-module of rank $r$, and let $\Delta \subseteq V$ be a
vanishing lattice. Then it is possible to find $r$ linearly independent elements
$\delta_1, \dotsc, \delta_r \in \Delta$, such that the group $\Gamma_{\{ \delta_1,
\dotsc, \delta_r \}}$ has finite index in $\Gamma_{\Delta}$.
\end{lemma}

\begin{proof}
Pick an arbitrary element $\delta_1 \in \Delta$. Since $\Delta$ is a vanishing
lattice, \cite{Janssen}*{Lemma~2.7} shows that $V$ is already generated by the
smaller set
\[
	\Delta_1 = \menge{\delta \in \Delta}{\text{$\pair{\delta_1}{\delta} = 1$ or
		$\delta = \delta_1$}}.
\]
We can therefore find $r$ linearly independent elements $\delta_1, \dotsc, \delta_r
\in \Delta$ that satisfy $\pair{\delta_1}{\delta_i} = 1$ for $i \geq 2$. Let $V'
\subseteq V$ be their span; then $V'$ is free of rank $r$, and the quotient is $V/V'$
is finite, say of order $k$.

Let $T_i = T_{\delta_i}$ be the corresponding transvections, and $\Gamma' =
\Gamma_{\{ \delta_1, \dotsc, \delta_r \}}$ the group generated by them. Note that
each $T_i$ actually preserves $V'$, which allows us to think of $\Gamma'$ as a
subgroup of $\Sp(V')$ where convenient. For $i \geq 2$, we have
\[
	T_i T_1 ( \delta_i ) = T_i ( \delta_i + \delta_1 ) = 
	\delta_i + \delta_1 - \pair{\delta_i + \delta_1}{\delta_i} \delta_i = \delta_1,
\]
and so all the $\delta_i$ lie in one orbit of the group $\Gamma'$. It follows that
$\Delta' = \Gamma' \cdot \{ \delta_1, \dotsc, \delta_r \}$ is itself a vanishing lattice in
$V'$. Theorem~\ref{thm:Janssen}, applied to $\Delta' \subseteq V'$, shows that $\Gamma'$
contains the subgroup 
\[
\begin{split}
	\Sptsh(V') = \menge{h \in \Sp(V')}{&\text{for any $\lambda' \in \Hom(V', \ZZ)$, there
		exists $v' \in V'$} \\
	&\text{such that $\lambda'(hx - x) = 2\pair{v'}{x}$ for all $x \in V'$}}.
\end{split}
\]

Now let $g \in \Spsh(V)$ be an arbitrary element. We shall prove that a fixed power
of $g$ preserves $V'$ and belongs to $\Sptsh(V')$.
In particular, this will show that $\Gamma'$ has finite index in $\Gamma$.

To begin with, let us assume that $g(V') = V'$. We then need to find a fixed integer
$m \geq 1$, such that $g^m$ belongs to $\Sptsh(V')$. In other words, given any
$\lambda' \in \Hom(V', \ZZ)$, there should exist a vector $v' \in V'$ such that
$\lambda'(g^m x - x) = 2\pair{v'}{x}$.
So let $\lambda' \in \Hom(V', \ZZ)$ be an arbitrary linear functional. 
Since $V/V'$ is of order $k$, the same is true for the group $\Ext^1(V/V', \ZZ)$.
The exact sequence
\begin{diagram}
	0 &\rTo& \Hom_{\ZZ}(V, \ZZ) &\rTo& \Hom_{\ZZ}(V', \ZZ) &\rTo& 
		\Ext_{\ZZ}^1 (V/V', \ZZ) &\rTo& 0
\end{diagram}
shows that the multiple $\lambda = k \lambda'$ extends uniquely to a linear
functional on all of $V$. By assumption, $g$ is an element of $\Spsh(V)$, and so
there is a vector $v \in V$ with the property that $\lambda(gx - x) = \pair{v}{x}$
for all $x \in V$. Let $w = k v \in V'$; then we have in particular that
\begin{equation} \label{eq:gx}
	k^2 \lambda'(gx - x) = \pair{w}{x}
\end{equation}
for all $x \in V'$.  After repeated application of \eqref{eq:gx}, we obtain
\begin{equation} \label{eq:gmx}
\begin{split}
	k^2 \lambda'(g^m x - x) & = \pair{w}{x + gx + \dotsb + g^{m-1} x} \\
		&= \pair{w + g^{-1} w + \dotsb + g^{-m+1} w}{x}.
\end{split}
\end{equation}
We thus need to choose $m$ so that $w + g^{-1} w + \dotsb + g^{-m+1} w \in 2 k^2 V'$.

The quotient $V' / 2 k^2 V'$ has finite order. By the Pigeon Hole Principle, there
exists a fixed integer $p \geq 1$ such that $g^{-p} w \equiv w \mod 2 k^2 V'$. Taking
$m = 2k^2 p$, we see that
\[
	w + g^{-1} w + \dotsb + g^{-m+1} w \equiv
	2k^2 \bigl( w + g^{-1} w + \dotsb + g^{-p+1} w \bigr) \equiv 0
	\mod 2 k^2 V'.
\]
Let $v' \in V'$ be any vector for which $w + g^{-1} w + \dotsb + g^{-m+1} w = 2 k^2
v'$. Then \eqref{eq:gmx} shows that
\[
	k^2 \lambda'(g^m x - x) = 2 k^2 \pair{v'}{x}.
\]
Since this is an identity in $\ZZ$, it follows that $\lambda'(g^m x - x) = 2
\pair{v'}{x}$ for every $x \in V'$. In other words, $g^m \in \Sptsh(V')$, which is
the result we were after, but with the extra assumption that $g(V') = V'$.

We now consider an arbitrary element $g \in \Spsh(V)$. Then $g(V') \subseteq V$ is
again a submodule of index $k$; since there are only finitely many such, we can find a
fixed positive integer $q$ such that $g^q(V') = V'$. The preceding analysis now applies,
and shows that $g^{mq} \in \Sptsh(V')$. Consequently, the index of $\Gamma'$ in
$\Gamma$ is at most $mq$.
\end{proof}

\subsection*{Injectivity of the restriction map}

In the presence of a vanishing lattice, it is again possible to prove the injectivity
of the restriction map by a fairly simple argument. The most convenient setting is the
following. Let $V$ be a free $\ZZ$-module of finite
rank, with an alternating bilinear form $B \colon V \tensor V \to \ZZ$. Let $G$ be a finitely
generated group acting on $V$, and assume that there are generators $g_1, \dotsc,
g_n$ for $G$, and distinguished elements $e_1, \dotsc, e_n$ of $V$, such that
\[
	g_i v = v - \pair{v}{e_i} e_i = T_{e_i}(v)
\]
for all $i$. Furthermore, assume that $\Delta = G \cdot \{ e_1, \dotsc, e_n \}$
is a vanishing lattice in $V$. Of course, the image of $G$ in $\Spsh(V)$ is
then exactly the monodromy group $\Gamma_{\Delta}$.

\begin{proposition} \label{prop:group-even}
Let $V$ be a free $\ZZ$-module of finite rank with a $G$-action, subject to the 
assumptions above. Let $V_{\QQ} = V \tensor \QQ$. Then the restriction map
\[
	H^1 \bigl( G, V_{\QQ} \bigr) \to \prod_{g \in G} V_{\QQ} / (g - \id) V_{\QQ}
\]
is injective.
\end{proposition}

\begin{proof}
Let $\phi \in Z^1 \bigl( G, V_{\QQ} \bigr)$ represent an element of the kernel; we
have to show that it belongs to $B^1 \bigl( G, V_{\QQ} \bigr)$. 
To simplify the notation, we shall write $\Gamma = \Gamma_{\Delta}$. We begin the
proof by noting that $\phi$ is identically zero on the normal subgroup $N = \ker(G
\to \Gamma)$. Indeed, given any $g \in G$, we can find some $v \in V_{\QQ}$ such that
$\phi(g) = gv - v$; thus any $g$ that acts trivially on $V$ automatically satisfies
$\phi(g) = 0$. Consequently, $\phi$ descends to an element in $H^1 \bigl( \Gamma,
V_{\QQ} \bigr)$. Since $H^1 \bigl( \Gamma, V_{\QQ} \bigr)$ is easily seen to inject
into $H^1 \bigl( G, V_{\QQ} \bigr)$, we may assume from now on that we are dealing
with an element $\phi \in Z^1 \bigl( \Gamma, V_{\QQ} \bigr)$.

Let $r = \dim V_{\QQ}$. Using Lemma~\ref{lem:finite-index}, we can find $r$ linearly
independent elements $\delta_1, \dotsc, \delta_r \in V$, such that $\Gamma' =
\Gamma_{\{ \delta_1, \dotsc, \delta_r \}}$ has finite index in $\Gamma$. Let $T_i =
T_{\delta_i} \in \Gamma'$. As in the odd case, we can adjust $\phi$ by an element of
$B^1 \bigl( \Gamma, V_{\QQ} \bigr)$ to make sure that $\phi(T_r \dotsm T_1) = 0$. By
assumption, we can also find vectors $v_i \in V_{\QQ}$ such that
\[
	\phi(T_i) = T_i v_i - v_i = - \pair{v_i}{\delta_i} \delta_i = a_i v_i,
\]
for $a_i = - \pair{v_i}{\delta_i}$. An application of Lemma~\ref{lem:expand} shows that
\[
	0 = \phi(T_r \dotsm T_1) = \sum_{k=1}^r b_k \delta_k,
\]
for coefficients $b_k \in \QQ$ satisfying the relations in \eqref{eq:relations}.
Since the $\delta_i$ are linearly independent, we have $b_k = 0$ for all
$k$, and thus $a_i = 0$ for all $i$. After the adjustment, the cocycle $\phi$ thus 
satisfies $\phi(T_i) = 0$ for all $i$. Since $\Gamma'$ is generated by the
transvections $T_i$, we conclude that $\phi(\Gamma') = 0$.

It is now easy to show that $\phi$ is identically zero. Let $m$ be the index of $\Gamma'$
in $\Gamma$. Take an arbitrary element $\delta \in \Delta$. As usual, there is a
vector $w \in V_{\QQ}$ such that
\[
	\phi(T_{\delta}) = T_{\delta} w - w = - \pair{w}{\delta} \delta.
\]
From this, one easily deduces that
\[
	\phi \bigl( T_{\delta}^m \bigr) = -m \pair{w}{\delta} \delta = 
		 m \phi \bigl( T_{\delta} \bigr).
\]
On the other hand, $T_{\delta}^m$ belongs to $\Gamma'$, and so $\phi \bigl( T_{\delta}^m
\bigr) = 0$. Since the $T_{\delta}$ together generate $\Gamma$, and $\phi$ is a
cocycle, we then have $m \phi = 0$, and hence $\phi = 0$. This proves the assertion.
\end{proof}

\subsection*{Conclusion of the argument}

To conclude that \eqref{eq:step3} is injective, we now apply our general result to
the vanishing cohomology $V_{\QQ} = H^{d-1}(S_0, \QQ) \van$. All the assumptions are
satisfied by Section~\ref{sec:Lefschetz}, if we let $\pair{u}{v} = \eps_d (u,v)$ be a
multiple of the intersection pairing on $S_0$.

In more detail, we set $V = H^{d-1}(S_0, \ZZ) \van$ modulo torsion; it is generated
by the vanishing cycles $e_1, \dotsc, e_n$ of any Lefschetz pencil. Let $g_1, \dotsc,
g_n$ be the corresponding elements of the fundamental group $G$. The collection of
all vanishing cycles $\Delta = G \cdot \{ e_1, \dotsc, e_n \}$ is then a vanishing
lattice in $V$, because $e_1, \dotsc, e_n$ all lie in one $G$-orbit, and because of
Lemma~\ref{lem:number-one}.  Proposition~\ref{prop:group-even} now gives us the
injectivity of the restriction map in \eqref{eq:inj-even}, which finishes the proof
that \eqref{eq:map-main} is injective when the dimension of $X$ is even.

\section{Completing the proof of the main theorem}
\label{sec:conclusion}

In this section, we finish the proof of Theorem~\ref{thm:tube-main} about the
surjectivity of the tube mapping. As we have seen in Section~\ref{sec:proof}, it is
sufficient to show that the dual mapping \eqref{eq:tube-dual} is injective. We have
already proved that the map in \eqref{eq:map-main} is injective for all values of
$d$; it remains to convince ourselves that the two maps are the same. This is almost
obvious; but for the sake of completeness, a proof is included here.

\subsection*{Smooth families over the circle}

We first consider the tube mapping in the case of a family of smooth manifolds over
$\SSn{1}$. So let $f \colon Y \to \SSn{1}$ be a proper and submersive map of smooth
manifolds. Let $m$ be the real dimension of $Y$, and let $Y_0$ be the fiber of $f$
over the point $1 \in \SSn{1}$. The cohomology of $Y$ and $Y_0$ can be represented
by smooth differential forms, and this will be done throughout.

The fundamental group of $\SSn{1}$ is isomorphic to $\ZZ$; it acts by monodromy on 
the homology and cohomology of $Y_0$.  This action is easily described.  Namely, the
smooth map $e \colon \RR \to \SSn{1}$, $t \mapsto \exp(2 \pi i t)$, makes $\RR$ into
the universal covering space of $\SSn{1}$. Since $Y$ is a smooth fiber bundle over
$\SSn{1}$, its pullback to $\RR$ is diffeomorphic to $\RR \times Y_0$.  The following
diagram shows the relevant maps.
\begin{diagram}[width=2.5em,midshaft]
	Y_0 &\rInto^{i_t}& Y_0 \times \RR &\rTo^{\Phi}& Y \\
	\dTo && \dTo && \dTo^f \\
	\{t\} &\rInto& \RR &\rTo^e& \SSn{1}	
\end{diagram}
We also write $F = \Phi \circ i_1$, which is a diffeomorphism from $Y_0$ to itself.
Note that $\Phi \circ i_0$ is simply the inclusion of $Y_0$ into $Y$.

Given a homology class $\alpha \in H_i(Y_0, \QQ)$, a flat translate of $\alpha$ to
the fiber over $e(t)$ is given by $(\Phi \circ i_t)(\alpha)$. In particular,
the monodromy action $T_i \colon H_i(Y_0, \QQ) \to H_i(Y_0, \QQ)$ by the standard
generator is $T_i(\alpha) = (\Phi \circ i_1)_{\ast} \alpha = F_{\ast} \alpha$. 
Similarly, we have an action on cohomology $T^i \colon H^i(Y_0, \QQ) \to H^i(Y_0, \QQ)$.

Tube classes on $Y$ are defined in the following way.  Suppose that $\alpha \in \ker
T_{k-1}$ is a monodromy-invariant homology class on $Y_0$. This means that there is a
$k$-chain $A$ on $Y_0$, such that $\partial A = F(\alpha) - \alpha$.  Translating
$\alpha$ flatly along $\SSn{1}$ and taking the trace in $Y$ gives the $k$-chain
$\Gamma = \Phi \bigl( \alpha \times \unitint \bigr)$. Then $\Gamma - A$ is closed,
and its class $\tau(\alpha) \in H_k(Y, \QQ)$ is the tube class determined by
$\alpha$. Of course, $\tau(\alpha)$ is only defined up to elements of
$H_d(Y_0, \QQ)$, because of the ambiguity in choosing $A$.

We now have to connect this topological construction with the one coming from the
Leray spectral sequence for the map $f$. The latter degenerates at $E_2$, and gives us for
each $k \geq 0$ a short exact sequence
\begin{equation} \label{eq:seq-S1}
\begin{diagram}
	0 &\rTo& H^1 \bigl( \SSn{1}, R^{k-1} \fl \QQ \bigr) &\rTo&
		H^k(Y, \QQ) &\rTo& H^0 \bigl( \SSn{1}, R^k \fl \QQ \bigr) &\rTo& 0.
\end{diagram}
\end{equation}
The first and third group can be computed explicitly; we have
\[
	H^1 \bigl( \SSn{1}, R^{k-1} \fl \QQ \bigr) \simeq \coker T^{k-1} \qquad
		\text{and} \qquad
	H^0 \bigl( \SSn{1}, R^k \fl \QQ \bigr) \simeq \ker T^k.
\]
Now suppose we are given a cohomology class in $H^k(Y, \QQ)$ whose restriction to
the fibers of $f$ is trivial. By virtue of \eqref{eq:seq-S1}, it defines a class in
$H^1 \bigl( \SSn{1}, R^{k-1} \fl \QQ \bigr)$, and hence in $\coker T^{k-1}$. The
following lemma gives a formula for this class.

\begin{lemma} \label{lem:tube-S1}
Let $\beta$ be a smooth and closed $k$-form on $Y$, representing an element of
$\ker \bigl( H^k(Y, \QQ) \to H^k(Y_0, \QQ) \bigr)$. Choose any
$(k-1)$-form $\gamma$ on $Y_0 \times \RR$ with $\Phiu \beta = d \gamma$, and let
$\gamma_t = \itu \gamma$. 
\begin{enumerate}[label=(\roman{*}), ref=(\roman{*})]
\item The element of $\coker T^{k-1}$ determined by $\beta$ is 
$\lambda(\beta) = (F^{-1})^{\ast} \gamma_1 - \gamma_0$. \label{enum:lem-i}
\item For every monodromy-invariant class $\alpha \in H^{k-1}(Y_0, \QQ)$, we have
\[	
	\int_{\tau(\alpha)} \beta = \int_{\alpha} \lambda(\beta).
\]
where $\tau(\alpha)$ is the tube class on $Y$ coming from $\alpha$.
\label{enum:lem-ii}
\end{enumerate}
\end{lemma}
\begin{proof}
Note that $\lambda(\beta)$ is a closed $(k-1)$-form on $S_0$. 
Let $m$ be the dimension of the smooth manifold $Y$, and let $i \colon Y_0 \to Y$ be
the inclusion map. For every closed $(m-k)$-form $\omega$ on $Y$, a simple
calculation using Stokes' Theorem shows that
\begin{align*}
	\int_Y \beta \wedge \omega 
	&= \int_{Y_0 \times \unitint} \Phiu \beta \wedge \Phiu \omega
	= \int_{Y_0 \times \unitint} d \gamma \wedge \Phiu \omega \\
	&= \int_{Y_0 \times \unitint} d \bigl( \gamma \wedge \Phiu \omega \bigr)
	= \int_{Y_0} i_1^{\ast} \bigl( \gamma \wedge \Phiu \omega \bigr) -
		\int_{Y_0} i_0^{\ast} \bigl( \gamma \wedge \Phiu \omega \bigr) \\
	&= \int_{Y_0} \gamma_1 \wedge F^{\ast} \bigl( \iu \omega \bigr)
			- \int_{Y_0} \gamma_0 \wedge \iu \omega
	= \int_{Y_0} \bigl( (F^{-1})^{\ast} \gamma_1 - \gamma_0 \bigr) \wedge
		\iu \omega.
\end{align*}
The assertion in \ref{enum:lem-i} now follows by duality.

To prove the second half, we recall that the tube class is given by $\Gamma - A$,
where $\partial A = F(\alpha) - \alpha$ on $Y_0$, and $\Gamma = \Phi \bigl( \alpha \times
\unitint \bigr)$. Again using Stokes' Theorem, we compute that
\begin{align*}
	\int_{\alpha} (F^{-1})^{\ast} \gamma_1 - \gamma_ 0 
	&= \int_{\alpha} (\gamma_1 - \gamma_0) + 
		\int_{F^{-1}(\alpha) - \alpha} \gamma_1
	= \int_{\partial(\alpha \times \unitint)} \gamma -
		\int_{F^{-1}(\partial A)} i_1^{\ast} \gamma  \\
	&= \int_{\alpha \times \unitint} d \gamma -
		\int_{F^{-1}(A)} i_1^{\ast}(d \gamma) 
	= \int_{\alpha \times \unitint} \Phiu \beta - 
		\int_{F^{-1}(A)} i_1^{\ast} \bigl( \Phiu \beta \bigr) \\
	&= \int_{\Gamma} \beta - \int_{F^{-1}(A)} F^{\ast} \beta 
	= \int_{\Gamma - A} \beta = \int_{\tau(\alpha)} \beta.
\end{align*}
This is the identity asserted in \ref{enum:lem-ii}.
\end{proof}

\subsection*{Identity of the two maps}

We are now ready to show that the map in \eqref{eq:map-main} is equal to the dual of
the tube mapping in \eqref{eq:tube-dual}. Let $g \in G$ be an element of the
fundamental group of $\Psm$, and let $\alpha \in H_{d-1}(S_0, \QQ)$ be any class
invariant under the action by $g$. We write $\tau_g(\alpha) \in H_d(X, \QQ)$ for the
tube class determined by $\alpha$; as we saw, it is well-defined up to the addition
of elements in $H_d(S_0, \QQ)$.

Take any closed $d$-form $\omega$ on $X$, whose class lies in $H^d(X, \QQ) \prim$.
Under the mapping
\[
	H^d(X, \QQ) \prim \to 
		\prod_{g \in G} H^{d-1}(S_0, \QQ) \van / (g - \id) H^{d-1}(S_0, \QQ) \van.
\]
in \eqref{eq:map-main}, $\omega$ is sent to an element of the product with coordinates
$\bigl( \lambda_g(\omega) + \im(g - \id) \bigr)$. 
Of course, $\lambda_g(\omega)$ itself is not uniquely determined by $\omega$; we
choose this notation only because the ambiguity turns out not to matter.

To prove that the map in \eqref{eq:map-main} really is the dual of the tube mapping,
it suffices to establish the identity
\begin{equation} \label{eq:identity-g}
	\int_{\tau_g(\alpha)} \omega = \int_{\alpha} \lambda_g(\omega)
\end{equation}
To do this, represent $g$ by an immersion $\SSn{1} \to
\Psm$, and let $f \colon Y \to \SSn{1}$ be the pullback of the family $\pism \colon
\famXsm \to \Psm$. Then $Y$ is a smooth manifold of dimension $m = 2d-1$. We have the
following diagram of maps:
\begin{diagram}[width=2.5em,midshaft]
	Y &\rTo^h& \famXsm &\rTo^q& X \\
	\dTo^f && \dTo^{\pism} \\
	\SSn{1} &\rTo^g& \Psm
\end{diagram}	
The fiber over the base point of $\SSn{1}$ is $Y_0 = S_0$, in the
notation used above. Then $\alpha \in H_{d-1}(S_0, \QQ)$ determines a tube class
$\tau(\alpha)$ on $Y$, and by the definition of the tube mapping, we have
\[
	\tau_g(\alpha) \equiv (qh)_{\ast} \tau(\alpha) \mod H_d(S_0, \QQ).
\]
Since $\omega$ is primitive, its restriction to $S_0$ is trivial.
Lemma~\ref{lem:tube-S1}, applied to the class $(qh)^{\ast} \omega$, shows that
\[
	\int_{\tau_g(\alpha)} \omega = \int_{\tau(\alpha)} (qh)^{\ast} \omega
	= \int_{\alpha} \lambda \bigl( (qh)^{\ast} \omega \bigr).
\]
The class $\lambda \bigl( (qh)^{\ast} \omega \bigr)$ is determined by the Leray
spectral sequence for the map $f$. On the other hand, the class $\lambda_g(\omega)$ is
determined in exactly the same way by the Leray spectral sequence for $\pism$.
But both spectral sequences are compatible with each other, starting from the
$E_2$-page, and so it has to be the case that
\[
	\lambda \bigl( (qh)^{\ast} \omega \bigr)	\equiv (qh)^{\ast} \lambda_g(\omega) 
		\mod (g - \id) H^{d-1}(S_0, \QQ).
\]
Now $\alpha$ is $g$-invariant, and its integral against any element of
$(g - \id) H^{d-1}(S_0, \QQ)$ is therefore zero. It follows that
\[
	\int_{\alpha} \lambda \bigl( (qh)^{\ast} \omega \bigr) = 
		\int_{\alpha} (qh)^{\ast} \lambda_g(\omega) = 
		\int_{\alpha} \lambda_g(\omega).
\]
After combining this with the other equality, we obtain \eqref{eq:identity-g}. 

\section{Additional thoughts and questions}

This section is a collection of several additional thoughts and questions.

\subsection*{A vanishing theorem}

Let $V = H^{d-1}(S_0, \ZZ) \van$ modulo torsion, and let $\Gamma$ be the image of the
monodromy representation $G \to \Aut(V)$. Also let $N = \ker(G \to \Gamma)$ be its
kernel. From the Hochschild-Serre spectral sequence \cite{Weibel}*{p.195} for the
normal subgroup $N$, we get an exact sequence
\begin{equation} \label{eq:Hochschild}
\begin{diagram}
	0 &\rTo& H^1(\Gamma, V) &\rTo& H^1(G, V) &\rTo&
		H^1(N, V)^{\Gamma} &\rTo^{\partial}& H^2(\Gamma, V).
\end{diagram}
\end{equation}
An element of the group $H^1(N, V)$ is simply a homomorphism from $\Nab = N/\lbrack N, N
\rbrack$ to $V$, because the action of $N$ on $V$ is trivial. The group $\Gamma$
naturally acts on $\Nab$ by conjugation, and we have
\[
	H^1(N, V)^{\Gamma} = \menge{\psi \in \Hom \bigl( \Nab, V \bigr)}
		{\text{$\psi(gxg^{-1}) = g \psi(x)$ for all $g \in \Gamma$}}.
\]

When the degree of the embedding of $X$ into projective space is sufficiently large,
and $d \geq 2$, we have $H^d(X, \QQ) \prim \simeq H^1 \bigl( G, V_{\QQ} \bigr)$ by
Proposition~\ref{prop:Nori}. We shall assume both things from now on.

\begin{lemma} \label{lem:phi-delta}
Let $\phi \in Z^1 \bigl( \Gamma, V_{\QQ} \bigr)$ be an arbitrary cocycle. For any $\delta
\in \Delta$, we have $\phi(T_{\delta}) \in \QQ \cdot \delta$.
\end{lemma}

\begin{proof}
Since all elements of $\Delta$ are in the same $G$-orbit, it suffices to prove the
statement for a single vanishing cycle. It is most convenient to take $\delta$ equal
to one of the vanishing cycles $e_i$ in a Lefschetz pencil, corresponding to a
hyperplane section $S_i$ with a single ordinary double point. By
Lemma~\ref{lem:Si-trivial}, every primitive cohomology class has trivial restriction
to $S_i$, and this implies that all $\psi \in Z^1 \bigl( G, V_{\QQ} \bigr)$ satisfy
$\psi(g_i) \in \QQ e_i$. But $H^1 \bigl( \Gamma, V_{\QQ} \bigr)$ is a subgroup of
$H^1 \bigl( G, V_{\QQ} \bigr)$, and the result follows.
\end{proof}

\begin{proposition}
We have $H^1 \bigl( \Gamma, V_{\QQ} \bigr) = 0$.
\end{proposition}

\begin{proof}
We will prove this for even values of $d$, using the results from
Section~\ref{sec:group-even}. Referring to the work of W.~Ebeling \cite{Ebeling} on
monodromy groups of symmetric vanishing lattices, a similar argument should work in
the odd case.

So let $\phi \in Z^1 \bigl( \Gamma, V_{\QQ} \bigr)$ be a cocycle. As in the proof of
Proposition~\ref{prop:group-even}, we can find $r$ linearly independent elements
$\delta_1, \dotsc, \delta_r \in \Delta$, where $r = \dim V_{\QQ}$, such that $\Gamma'
= \Gamma_{\{ \delta_1, \dotsc, \delta_r \}}$ has finite index in $\Gamma$. Let $T_i =
T_{\delta_i}$. By Lemma~\ref{lem:phi-delta}, we have $\phi(T_i) = a_i \delta_i$ for
certain $a_i \in \QQ$. Now the pairing $B$ is nondegenerate, because it is a multiple
of the intersection pairing, and so we can find a vector $v \in V_{\QQ}$ with the
property that $\pair{v}{\delta_i} = a_i$ for all $i$. Adjusting $\phi$ by the
coboundary $(g \mapsto gv - v)$, we achieve that $\phi(T_i) = 0$ for all $i$, hence
that $\phi(\Gamma') = 0$. This implies that $\phi(\Gamma) = 0$. Since $\Gamma'$ has
finite index in $\Gamma$, we can then argue as before to show that $\phi = 0$.
\end{proof}

The exact sequence in \eqref{eq:Hochschild} now lets us conclude that
\[
	H^d(X, \QQ) \prim \subseteq
		\menge{\psi \in \Hom \bigl( \Nab, V_{\QQ} \bigr)}
			{\text{$\psi(gxg^{-1}) = g \psi(x)$ for all $g \in \Gamma$}}.
\]
This has two implications. On the one hand, it shows that the kernel of the monodromy
representation is a very large group. In fact, since $V_{\QQ}$ is an irreducible
$\Gamma$-module, we have $\im \psi = V_{\QQ}$ whenever $\psi \neq 0$. In the presence
of primitive cohomology classes on $X$, the abelianization of $N$ is therefore of
rank at least as big as that of the vanishing cohomology $H^{d-1}(S_0, \QQ) \van$.

On the other hand, only a small part of the right-hand side in \eqref{eq:map-main} is
required to detect the primitive cohomology of $X$.

\begin{lemma} \label{lem:N}
Let $\psi \in H^1(N, V)^{\Gamma} \tensor \QQ$, and $v \in V_{\QQ}$ be any nonzero
vector. If the element of $H^1(N, \QQ) = \Hom \bigl( \Nab, \QQ
\bigr)$ given by the rule $x \mapsto \Pair{\psi(x)}{v}$ is zero, then $\psi$
itself is zero.
\end{lemma}
\begin{proof}
Suppose the element in question is zero. Using the properties of $\psi$, we then have
\[
	\Pair{\psi(x)}{g v} = 
		\Pair{g^{-1} \psi(x)}{v} = \Pair{\psi(g^{-1} x g)}{v} = 0,
\]
for arbitrary $x \in \Nab$. But since $V_{\QQ}$ is an irreducible $G$-module, the 
orbit of $v$ generates $V_{\QQ}$, and so the identity implies that $\psi = 0$.
\end{proof}

As in the introduction, this can again be interpreted in terms of covering spaces.
Let $T$ be the covering space of
$\Psm$ corresponding to the subgroup $N \subseteq G$. Note that this is the smallest
covering space of $\Psm$ on which the local system $(R^{d-1} \pisml \QQ) \van$
becomes trivial. Since $N$ is a normal subgroup, $T$ is connected, and the fiber
over the point $H_0 \in \Psm$ is in bijection with the group $\Gamma$. For every
nonzero vector $v \in V_{\QQ}$, we get a map $H^d(X, \QQ) \prim \to H^1(T, \QQ)$, and
all these maps are injective by Lemma~\ref{lem:N}

\subsection*{Holomorphic disks}

Another question is whether one can find tube classes with additional
properties that still generate the primitive cohomology. For instance, one can
require that the element $g \in G$ be the boundary of a holomorphically immersed disk
in $P$. When the dimension of $X$ is odd, this is true. Indeed, looking at the proof of
Proposition~\ref{prop:group-odd}, we see that it is only necessary to test the
cocycle $\phi \in Z^1 \bigl( G, V_{\QQ} \bigr)$ on products of the form $g_1 \dotsm
g_m$, where each $g_i$ is the Picard-Lefschetz transformation corresponding to a
vanishing cycle in a fixed Lefschetz pencil. Obviously, every such product is
homotopic to the boundary of a holomorphically immersed disk in $\PPn{1}$. In the case of even
$d$, the proof does not allow this stronger conclusion.
\begin{question*}
Let $X$ be of even dimension. Is it true that the primitive cohomology of $X$ can be
generated by tube classes with $g$ equal to the boundary of a holomorphically
immersed disk in $P$?
\end{question*}

\section*{Acknowledgements}

This paper has its origins in a joint project with Herb Clemens, and it was Herb who
suggested the question to me. I take this opportunity to thank him for his generous
support during my time in graduate school, and in particular for many hours of
engaging and useful discussions. Two conversations with Claire Voisin were also very
helpful, and are gratefully acknowledged. Finally, my thanks go to Paul Taylor for
his \LaTeX-package \texttt{diagrams} that was used to typeset the commutative diagrams.

%++++++++++++++++++++++++++++++++++

\end{document}

%% file: abbreviations.tex
\newcommand{\eps}{\varepsilon}

\newcommand{\pair}[2]{\bigl\langle #1, #2 \bigr\rangle}

\newcommand{\tensor}{\otimes}

\newcommand{\Hom}[3][]{\operatorname{Hom}_{#1}(#2,#3)}
\newcommand{\Ext}[4][]{\operatorname{Ext}_{#1}^{#2}(#3,#4)}

\newcommand{\SSn}[1]{\mathbb{S}^{#1}}

\newcommand{\ZZ}{\mathbb{Z}}
\newcommand{\QQ}{\mathbb{Q}}
\newcommand{\RR}{\mathbb{R}}
\newcommand{\CC}{\mathbb{C}}

\newcommand{\PPn}[1]{\mathbb{P}^{#1}}

\newcommand{\menge}[2]{\bigl\{ \thinspace #1 \thinspace\thinspace \big\vert%
\thinspace\thinspace #2 \thinspace \bigr\}}

\newtheorem*{lem*}{Lemma}

\newtheorem*{thm*}{Theorem}
\newtheorem*{prop*}{Proposition}

\theoremstyle{definition}

\theoremstyle{remark}

\newtheorem*{ex*}{Example}

\theoremstyle{plain}

\DeclareMathOperator{\im}{im}

\DeclareMathOperator{\coker}{coker}

\DeclareMathOperator{\id}{id}

\newcommand{\define}[1]{\emph{#1}}

\newcommand{\transp}[1]{{#1}^{\dagger}}
\newcommand{\restr}[2]{#1\big\vert_{#2}}

\newcommand{\argbl}{\,\_\!\_\!\_\,}

\def\overbar#1#2#3{{%
	\setbox0=\hbox{\displaystyle{#1}}%
	\dimen0=\wd0
	\advance\dimen0 by -#2 
	\vbox {\nointerlineskip \moveright #3 \vbox{\hrule height 0.3pt width \dimen0}%
		\nointerlineskip \vskip 1.5pt \box0}%
}}